\documentclass[11pt, letterpaper,english]{amsart}
\usepackage[left=1in,right=1in,bottom=1in,top=1in]{geometry}
\usepackage{amsmath,amsthm,amssymb}
\usepackage{mathrsfs}
\usepackage[all,cmtip]{xy}
\usepackage{cite}
\usepackage{hyperref}
\usepackage{enumerate}
\usepackage{graphicx}
\usepackage[usenames,dvipsnames]{color}
\usepackage{colonequals} 

\theoremstyle{plain}
\newtheorem{lem}{Lemma}[section]
\newtheorem{lemma}[lem]{Lemma}

\newtheorem{theorem}[lem]{Theorem}

\newtheorem{proposition}[lem]{Proposition}

\newtheorem{corollary}[lem]{Corollary}

\newtheorem*{mtheorem}{Main Theorem}

\theoremstyle{definition}

\newtheorem{definition}[lem]{Definition}
\newtheorem{example}[lem]{Example}
\newtheorem{remark}[lem]{Remark}

\newtheorem{notation}[lem]{Notation}

\newcommand{\hh}{h}
\newcommand{\aaa}{\alpha}

\newcommand{\mathfont}{\mathbf}

\newcommand{\ZZ}{\mathfont Z}

\newcommand{\QQ}{\mathfont Q}

\newcommand{\FF}{\mathfont F}

\newcommand{\bP}{\mathbb{P}}

\newcommand{\cA}{\mathcal{A}}
\newcommand{\cB}{{\mathcal{A}^{\cF}}}

\newcommand{\cF}{\mathcal{F}}

\newcommand{\cM}{\mathcal{M}}

\DeclareFontFamily{OT1}{rsfs}{}
\DeclareFontShape{OT1}{rsfs}{n}{it}{<-> rsfs10}{}
\DeclareMathAlphabet{\mathscr}{OT1}{rsfs}{n}{it}

\DeclareMathOperator{\Aut}{Aut}

\newcommand{\Hod}{\mathbb{E}_g}
\newcommand{\Mgbar}{ \overline{\mathcal{M}}_g}

\newcommand{\beq}{ \begin{equation}}
\newcommand{\eeq}{ \end{equation}}

\newcommand{\calF}{\mathcal{F}}

\newcommand{\todo}[1]

\title{Mass formula for non-ordinary curves in one dimensional families}

\author{Renzo Cavalieri}
\address{Department of Mathematics, Colorado State University, Fort Collins, CO 80523, USA}
\email{renzo.cavalieri@colostate.edu}

\author{Rachel Pries}
\address{Department of Mathematics, Colorado State University, Fort Collins, CO 80523, USA}
\email{pries@colostate.edu}

\date{\today}

\begin{document} 
 
\begin{abstract}
This paper is about one dimensional families of cyclic covers of the projective line in positive characteristic.
For each such family, we study the mass formula for the number of non-ordinary curves in the family.
We prove two equations for the mass formula: 
the first relies on tautological intersection theory;
and the second relies on the $a$-numbers of non-ordinary curves in the family.
Our results generalize the Eichler--Deuring mass formula for supersingular elliptic curves;
they also generalize some theorems of Ibukiyama, Katsura, and Oort 
about supersingular curves of genus $2$ that have an automorphism of order $3$ or order $4$.
We determine the mass formula in many new cases, including 
linearized families of hyperelliptic curves of every genus and 
all families of cyclic covers of the projective line branched at four points.

keywords: curve, hyperelliptic curve, cyclic cover, Jacobian, mass formula, cycle class, tautological ring, Hodge bundle, 
intersection theory, Frobenius, non-ordinary, $p$-rank, $a$-number.

\noindent
MSC20: primary 11G20, 14C17, 14H10, 14H40, 14N35; secondary 11G10, 14G15, 14H37, 11M38 
\end{abstract}



\thanks{Cavalieri acknowledges support from Simons Collaboration Grant 420720 and NSF grant DMS - 2100962. 
Pries was partially supported by NSF grant DMS - 22-00418.  
The authors would like to thank John Voight and the anonymous referee for helpful conversations.}

\maketitle

\section{Introduction}

Suppose $p$ is a prime number and $k=\bar{\mathbb F}_p$.
If $E$ is an elliptic curve over $k$, 
we say that $E$ is \emph{ordinary} if the number of $p$-torsion points on $E$
is exactly $p$; if not, we say that $E$ is \emph{non-ordinary} (or \emph{supersingular}).
The Eichler--Deuring mass formula states that
\begin{equation} \label{EED}
\sum_{[E]} \frac{1}{\#{\mathrm{Aut}}(E)} = \frac{p-1}{24},
\end{equation}
where the sum is over the isomorphism classes of non-ordinary elliptic curves $E$ over $k$.

Thus the number of isomorphism classes of supersingular elliptic curves
grows linearly in $p$, with asymptotic rate of growth $(p-1)/12$.
From this, it is possible to determine the exact number of isomorphism classes of supersingular elliptic curves; 
for $p \geq 5$, this exact number is 
$\lfloor p/12 \rfloor + \epsilon_p$ where $\epsilon_p = 0,1,1,2$ when $p \equiv 1,5,7,11 \bmod 12$ respectively.
The exact number depends on the congruence of $p$ modulo $12$ because of the 
contributions from the elliptic curves with $j$-invariants $0$ and $1728$, which have extra automorphisms.

The Legendre family of elliptic curves is given by the affine equation
\begin{equation} \label{Legendre}
E_t: y^2=x(x-1)(x-t).
\end{equation}
One proof of the Eichler--Deuring mass formula
uses the Cartier operator and results of Igusa to show that there are $(p-1)/2$ values of $t$ for which 
$E_t$ is supersingular, see Remark~\ref{Rlegendre}.
For a second proof, 
one studies the divisor of zeroes of a section of $\mathrm{det}({\mathbb E})^{12}$, 
where ${\mathbb E}$ denotes the Hodge bundle; this divisor is supported
on the cusp at the boundary of the Legendre family.
Each of these proofs sheds light on the material covered by the other.

In this paper, we generalize the Eichler--Deuring mass formula to the setting of one dimensional 
families $\mathcal{F}$ of curves of genus $g$ that are cyclic covers of the projective line ${\mathbb P}^1$ over $k$. 
We consider the locus $V_{g-1}$ of non-ordinary abelian varieties, which is a codimension one cycle in the 
moduli space $\cA_g$ of principally polarized abelian varieties of dimension $g$, and its pullback 
to the moduli space $\overline{\cM}_g$ of stable curves of genus $g$.
We define the mass formula for $\cF$ as the degree of the zero dimensional class obtained by intersecting $V_{g-1}$ 
with the curve in $\overline{\cM}_g$ determined by the family $\cF$. 

When the generic curve in the family is ordinary,
this degree gives a weighted count of the non-ordinary curves in $\cF$.  Each curve $X$ is weighted with the usual reciprocal of the size of $\mathrm{Aut}(X)$, and also with the multiplicity of intersection between 
$\mathcal{F}$ and $V_{g-1}$ at $X$.  

In our first main result, Theorem~\ref{thm:secondeq}, we reinterpret 
the mass formula in terms of a tautological intersection number.
Because of this, the mass formula admits a natural evaluation using a
tautological intersection theory computation from \cite{COSlambda1}.  
In our second main result, Theorem~\ref{thm:firsteq}, we reinterpret 
the mass formula in terms of the $a$-numbers of the non-ordinary curves in the family.
 
We then determine the mass formula for important one dimensional families of curves.
These families generalize the Legendre family in several different ways and also
generalize earlier work in genus $2$ from \cite{IKO}.
Specifically, we determine the mass formula for non-ordinary curves for:
\begin{itemize}
\item Corollary~\ref{Chyp}: Linearized one dimensional families of hyperelliptic curves of genus $g$ for any $g \geq 2$. 
See Definition~\ref{Dlinear} for the definition of linearized.
\item Corollary~\ref{C4point}: Families of cyclic covers of $\mathbb{P}^1$ branched at $4$ points, for any degree $d$ and any inertia type.
\end{itemize}

We now give a more detailed introduction to the objects and results in the paper.

\subsection{Cyclic covers of the projective line}

Let $k$ be an algebraically closed field of characteristic $p$.
Let $X$ be a (smooth, connected, projective) curve of genus $g$ defined over $k$.  
Let $J_X$ be the Jacobian of $X$.

Let $d \geq 2$ be an integer with $p \nmid d$. 
We suppose that $X$
admits a cyclic degree $d$ cover $\hh:X \to {\mathbb P}^1$ with $n \geq 4$ branch points.  
A discrete invariant of $\hh$ is the \emph{inertia type}, which is an 
$n$-tuple $a=(a_1,a_2, \ldots, a_n)$ of integers $a_i$ with $0 < a_i < d$ for $1 \leq i \leq n$ and 
$\sum_{i=1}^n a_i \equiv 0 \bmod d$.
Then $X$ admits an affine equation of the form
\begin{equation} \label{Ecurvegenerala}
y^d = \prod_{i=1}^{n-1} (x-t_i)^{a_i}, 
\end{equation} 
where $t_1=0$, $t_2 =1$ and $t_3, \ldots, t_{n-1}$ are pairwise distinct elements of $k-\{0,1\}$.

More precisely, by Riemann's Existence Theorem, there is a unique 
smooth projective connected curve $X$ having affine equation \eqref{Ecurvegenerala}.
It is given by the normalization of this equation at its singularities.
The inertia type determines the genus $g$ of $X$, see Lemma~\ref{LRH}.

Let $\tau \in \mathrm{Aut}(X)$ be an automorphism of order $d$
such that the quotient curve $X/\langle \tau \rangle$ is ${\mathbb P}^1$.
With respect to \eqref{Ecurvegenerala}, we can choose $\tau: (x,y) \mapsto (x, \zeta_d y)$, 
where $\zeta_d$ is a primitive $d$th root of unity.
Let ${\it \mathrm{Aut}(X, \tau)}$ denote the normalizer of $\tau$ in $\mathrm{Aut}(X)$.

\subsection{Non-ordinary curves}

The curve $X$ is \emph{ordinary} if the number of $p$-torsion points in $J_X(k)$ is exactly $p^g$.
The ordinary condition is equivalent to the action of $V$ on $H^0(X, \Omega^1)$ being invertible, where 
$V$ denotes the Verschiebung morphism.
If $X$ is not ordinary, then its $a$-number $a(X)$ is positive;
 the $a$-number is the co-rank of $V$ on $H^0(X, \Omega^1)$. 
Equivalent definitions are described in Section~\ref{Sanumber}.

There are some conditions on the inertia type and the congruence of $p$ modulo $d$
that are necessary for the curve $X$ in \eqref{Ecurvegenerala} to be ordinary for a 
generic choice of $\vec{t}=(t_3, \ldots, t_{n-1})$. 
Under mild hypotheses (including the case that $p$ is large and the case that
$p \equiv \pm 1 \bmod d$ and the case that $n=4$),
Bouw proved that these necessary conditions are also sufficient, \cite{Bouw}.
More generally, the inertia type and the congruence of 
$p$ modulo $d$ place restrictions on $a(X)$. 

\subsection{Main result}
We fix the degree $d$ and the inertia type $a$.
Let ${\mathcal A}_{d,a}$ denote the moduli space of  admissible $\mu_d$-covers of a rational curve
with inertia type $a$, as introduced in Section \ref{sec:msac}.  
A one dimensional family $\mathcal{F}$ of such covers over the field $k$ is, by definition, 
a morphism $\phi_\cF: \cB \to {\mathcal A}_{d,a}$ 
for some smooth, proper, irreducible, one dimensional Deligne--Mumford (DM)
stack $\cB \to {\mathrm{Spec}}(k)$.
When the general target curve of $\cF$ is smooth, we abbreviate this as a {\it one dimensional family of 
$\mu_d$-covers of $\mathbb{P}^1$}.

Let ${\overline{\mathcal M}}_g$ be the moduli space of stable curves of genus $g$.
There is a forgetful morphism 
\begin{equation} \label{Edefphi}
\phi^{\mathcal{F}}: \cB \to {\overline{\mathcal M}}_g,
\end{equation}
that records the source curve of the cover. Let $\mathcal{M}^{\mathcal{F}}$ be the image of $\phi^{\mathcal{F}}$ 
 and ${\delta^{\mathcal{F}}}$ be its degree.

Let $\mathbb{E}_g\to  {\overline{\mathcal M}}_g$ denote the Hodge bundle.
If $X$ is a curve of genus $g$, the sections of $\mathbb{E}_g$ over $X$ are the holomorphic $1$-forms on $X$.  
Let $\lambda_1$ denote its first Chern class. For a one dimensional family $\mathcal{F}$ of $\mu_d$-covers of $\mathbb{P}^1$ with inertia type $a$,  we denote by ${\mathrm{deg}}^{\cF}(\lambda_1)$ the degree of the pullback of $\lambda_1$ via  the morphism $\phi^{\mathcal{F}}$.

For a prime $p \nmid d$, consider the characteristic $p$ fiber of $\overline{\mathcal{M}}_{g}$.
The locus $V_{g-1}$ of non-ordinary curves
is a codimension one cycle determining a degree one class in the Chow ring of $\Mgbar$.  
We define the {\it mass formula} for the family $\mathcal{F}$ to be the intersection number
\begin{equation} \label{def:mass}\mu(\cF,p) =\int_{\Mgbar}[{\mathcal M}^{\cF}] \cdot [V_{g-1}] .\end{equation}
When the generic source curve in $\mathcal{F}$ is ordinary, we provide two ways to evaluate $\mu(\cF, p)$.
The first is in terms of {the intersection number} ${\mathrm{deg}}^{\cF}(\lambda_1)$ and the degree
$\delta^{\cF}$. 
The second is by a weighted count of the non-ordinary curves in $\cF$. 
Each curve $X$ is weighted by the multiplicity $m_X$ of the intersection of ${\mathcal M}^{\cF}$ 
and $V_{g-1}$ at $X$, and by the reciprocal of $\#\mathrm{Aut}(X)$.
Using Proposition~\ref{PmxF}, we rewrite the second formula in terms of other invariants: 
$\#\mathrm{Aut}(X, \tau)$, which is easier to compute than $\#\mathrm{Aut}(X)$;  
and the order of vanishing $\aaa_X$ of the determinant of the Cartier operator at $X$, which is closely 
related with the $a$-number $a(X)$.


We collect both results in the following statement.

\begin{mtheorem} \label{Tintro}
Let $\cF$ be a one dimensional family of admissible $\mu_d$-covers of a rational curve with inertia type $a$. 
Suppose $p\nmid d$ is a prime such that the generic point of the characteristic $p$ fiber 
${\mathcal M}^{\cF}_{p}$ of ${\mathcal M}^{\cF}$ represents an ordinary curve. 
Then
\begin{equation} \label{Eintro}
\sum_{[X]} \frac{\aaa_X}{\#{\mathrm{Aut}}(X, \tau)}
\stackrel{{Thm.\ \ref{thm:firsteq}}}{=} 
\mu(\cF,p)
\stackrel{{Thm.\ \ref{thm:secondeq}}}{=} (p-1) \frac{{\mathrm{deg}}^{\cF}(\lambda_1)}{\delta^{\cF}},
\end{equation}
where the sum on the left is over the isomorphism classes of non-ordinary curves $X$ in 
${\mathcal M}^{\cF}_{p}$. 
Furthermore, $\aaa_X \geq a(X)$, with equality when $n=4$ and $p \equiv 1 \bmod d$. 
\end{mtheorem}

In Lemma~\ref{Llegendre}, we show that \eqref{Eintro} specializes to the Eichler--Deuring mass formula for the Legendre family, where $d=2$, $n=4$, and $a=(1,1,1,1)$.
In Sections~\ref{SikoA} and \ref{SikoB}, we show that \eqref{Eintro} agrees with the results of 
Ibukiyama, Katsura and Oort \cite{IKO} for two families of curves of genus $2$.

The proofs of Theorems~\ref{thm:secondeq} and \ref{thm:firsteq} are in Section~\ref{SproofmaintheoremF}. 
In \cite{COSlambda1}, Cavalieri, Owens, and Somerstep found an explicit formula 
 for the class $\lambda_1$ on 
${\mathcal A}_{d,a}$ as a linear combination of boundary divisors; see Section~\ref{SCOS}.  
This allows us to evaluate the right hand side of \eqref{Eintro} for many families $\cF$. 
With additional information about the $a$-numbers,
we can then deduce information about the number of non-ordinary curves in $\cF$.

We note that Theorem~\ref{thm:secondeq} holds in greater generality than stated in this paper, see Remark~\ref{Rgeneral}.

\subsection{Corollaries}

Section~\ref{Slinhyp} contains our first corollary which 
is for linearized families of hyperelliptic curves of every genus $g \geq 2$ (see Definition \ref{Dlinear}).  

\begin{corollary} [Corollary~\ref{Chyp}]
Let $g \geq 2$ and $p$ be odd.  
Suppose $h(x) \in k[x]$ is a separable polynomial of degree $2g+1$ whose set of roots is not stabilized by any 
$\mathrm{PGL}_2(k)$ symmetry.
Suppose $p$ is a prime such that the generic curve in the characteristic $p$ fiber of 
$\cF_{h(x)} : y^2 = h(x)(x-t)$ is ordinary. 
Then the mass formula is $\mu(\cF_{h(x)}, p) = (p-1)g/4$.
\end{corollary}

For the proof of Corollary~\ref{Chyp}, we
study the intersection of $\cF_{h(x)}$ with the boundary of $\overline{\mathcal{M}}_g$.
This calculation is combinatorial and independent of the prime $p$.

In the second corollary, Corollary~\ref{C4point}, we determine the mass formula for any family of 
$\mu_d$-covers of $\mathbb{P}^1$ branched at four points.
This is a natural class of examples to consider because $\cF$ coincides with a one dimensional moduli space 
$\mathcal{A}_{d,a}$. 
In this case, the right hand side of \eqref{Eintro} can be computed explicitly using 
\cite[Theorem~1.2]{COSlambda1} and Lemma~\ref{lem:deg}.

For example, for the family $\cF$ of curves of genus $d-1$ given by:
\begin{equation} \label{E111}
X_t: y^d = x(x-1)(x-t),
\end{equation} 
the mass formula for the non-ordinary curves is $\mu(\cF, p)=(p-1) (d^2-1)/72 d^2$, 
if $d \geq 5$ with $\mathrm{gcd}(d,6) = 1$, and $p \equiv 1 \bmod d$, see Example~\ref{Ecase111}.

In Section~\ref{Sspecial}, we study $14$ families of cyclic covers associated with special Shimura varieties; 
these occur for curves of genus $1-7$, with the genus $1$ case being the Legendre family.
For these families, we have complete knowledge of the $a$-numbers of 
the non-ordinary curves.  Thus, in Corollary~\ref{Cmoonen}, in a direct generalization of the 
Eichler--Deuring formula, we determine 
both the mass formula and the asymptotic rate of growth 
of the number of non-ordinary curves in the 14 families as a linear function of the prime $p$. 
Here are two examples:

\begin{corollary} \label{Cintro}
If $d=5$ (resp.\ $d=7$), for $p \equiv 1 \bmod d$, 
the number of isomorphism classes of non-ordinary curves in the family \eqref{E111} grows 
linearly in $p$, with asymptotic rate of growth
$(p-1)/30$ (resp.\ $(p-1)/21$).
\end{corollary}

In Corollaries~\ref{CcaseB} and Corollary~\ref{Ccasesecond}, we determine the mass formula for 
families of curves of arbitrarily large genus whose automorphism group contains a dihedral group.
These results specialize to work in \cite{IKO} when $g=2$.

For most families of cyclic covers of $\mathbb{P}^1$, it is not feasible to 
obtain complete information about the $a$-numbers of the non-ordinary curves. 
Indeed, in certain cases, this is related to an open question about 
hypergeometric differential equations; see Remark~\ref{Rhypergeometric}.
The $a$-number is an invariant of the $p$-torsion group scheme (or de Rham cohomology)  
of a curve in characteristic $p$. 
Our results show that the $a$-numbers of curves in a family also distill subtle information about the 
geometry of the family in $\overline{\mathcal{M}}_g$.

This paper provides a new approach to studying non-ordinary curves in families using intersection theory.
We remark that it might be possible to generalize this approach to obtain results about curves with lower 
$p$-ranks, using the fact that the class of the $p$-rank $f$ strata in the Chow ring of the moduli space of curves
is known. 
The main issues are the following: one would need to work with a $g-f$ dimensional family of curves, which might 
involve more sophisticated geometry; classes of higher codimension 
make the intersection problem computationally more complex; and intersection theory does not provide information
when a family intersects a stratum that it is expected to miss for dimension reasons.

\section{Background}
\label{sec:background}

In this section, we provide some background material including: 
the $a$-number of a curve; cyclic covers of $\mathbb{P}^1$; moduli spaces of curves;
the Hodge bundle on $\overline{\mathcal{M}}_g$ and its first Chern class; and the cycle class of the non-ordinary
locus.
Much of this material is well known in one of the two communities that might be reading this work, 
but perhaps not in both.

For a smooth curve $X$, let $H^0(X, \Omega^1)$ denote the vector space of holomorphic one-forms.

\subsection{The $a$-number} \label{Sanumber}

Let $k$ be an algebraically closed field $k$ of characteristic $p$.
Suppose $A$ is a principally polarized abelian variety of dimension $g$ defined over $k$.
Let $A[p]$ denote its $p$-torsion group scheme.
In this paper, $A$ is usually the Jacobian $J_X$ of a smooth curve $X$ of genus $g$.

The \emph{$p$-rank} $s(A)$ of $A$ is the integer $s$ such that $\#A[p](k) =p^{s}$.
It is well-known that $0 \leq s \leq g$.
The abelian variety $A$ is \emph{ordinary} if and only if $s(A)=g$.

The \emph{$a$-number} of $A$ 
is $a(A): = \mathrm{dim}_k \mathrm{Hom}(\alpha_p, A[p])$ where the group scheme 
$\alpha_p$ is the kernel of Frobenius on the additive group ${\mathbb G}_a$.  
It equals the dimension of $\mathrm{Ker}(F) \cap \mathrm{Ker}(V)$
on the Dieudonn\'e module of $A[p]$, where $F$ is the Frobenius morphism and $V$ is the Verschiebung morphism.

If $A$ is not ordinary, then $a(A) > 0$, because there
is a non-trivial local-local summand of $A[p]$ on which $V$ is nilpotent.
More generally $0 < a(A) + s(A) \leq g$.

For a curve $X$, the $p$-rank $s(X)$ and the $a$-number $a(X)$ are that of its Jacobian.
The $p$-rank $s(X)$ is the stable rank of $V$ on $H^0(X, \Omega^1)$.
The $a$-number depends only on the rank of $V$;
by \cite[Equation~5.2.8]{LO}, $a(X) =g - \mathrm{rank}(V)$.

The action of $V$ on $H^0(X, \Omega^1)$ is given by the Cartier--Manin matrix.
As in \cite[Definitions~2.1, 2.1']{Yui}, this is   
the matrix for what is called the modified Cartier operator $C$. 
Here $C$ is a $1/p$-linear map which trivializes exact 1-forms and satisfies $C(f^{p-1}df) = df$ for any function $f$. 

\subsection{Cyclic covers of $\mathbb{P}^1$}
\label{sec:ccov}

Let $d \geq 2$.
Fix a primitive $d$-th root of unity $\zeta_d$ in $k$.  Let $\mu_d=\langle \zeta_d \rangle$.

Let $p$ be a prime with $p \nmid d$.  
The material in this subsection is well-known when working over $\mathbb{C}$, but holds 
equally well when working with the characteristic $p$ 
reduction of the curves.

\subsubsection{The inertia type}

An inertia type of length $n$ for $d$ is an $n$-tuple $a=(a_1,\ldots, a_n)$ 
with $0< a_i < d$
such that $\sum_{i=1}^n a_i \equiv 0 \bmod d$ and ${\rm gcd}(a_1,\ldots, a_n)=1$.

\begin{definition}
Let $\hh: X\to \mathbb{P}^1$ be a $\mu_d$-cover of $\mathbb{P}^1$.
Let $\tau\in \mathrm{Aut}(X)$ be the automorphism of order $d$ associated with $\zeta_d$ under
the inclusion $\mu_d \subset \mathrm{Aut}(X)$. 
Denote by $B$ a labeling of the branch points of $\hh$. 
Let $y$ be a general point of $\mathbb{P}^1$. 
For $1 \leq i \leq n$, let $\rho_i$ be a small loop based at $y$ winding once around the $i$-th branch point. 
Let $x$ be an inverse image of $y$.
If $\tilde{\rho}_{i,x}$ denotes the lift of $\rho_i$ based at $x$, then 
the end point of $\tilde{\rho}_{i,x}$ equals $\tau^{a_i}(x)$ for some integer $0< a_i<d$. 
The vector $a = (a_1, \ldots, a_n)$ is the inertia type of $\hh$.
\end{definition}

To define the inertia type of a $\mu_d$-cover over $k$, one can either lift the cover $h$ to characteristic $0$, or 
use Grothendieck's theorem about the prime-to-$p$ fundamental group of $\mathbb{P}^1-B$, or work exclusively over $k$ by 
replacing loops with isomorphisms of fiber functors.
 
\begin{remark} \label{Rswitch}
A labeling of the branch points is necessary to define the inertia type. 
Also, replacing $\tau$ by $\tau^\ell$ for any $\ell \in (\ZZ/d\ZZ)^\ast$ 
multiplies the inertia type of the $\mu_d$-cover by $\ell^{-1}$. 
The inertia type of a $\mu_d$-cover does not depend on the choice of the pre-image $x$ of $y$.
\end{remark}

\subsubsection{Isomorphisms and automorphisms}

Two $\mu_d$-covers with inertia type $a$ are {\it isomorphic} if there exists a commutative diagram
\beq
\xymatrix{X_1{\ar@(ul,ur)^{\tau_1}}   \ar[d]_{\hh_1} \ar[r]^\varphi & X_2{\ar@(ul,ur)^{\tau_2}}   \ar[d]^{\hh_2} \\
(\mathbb{P}^1,B_1) \ar[r]^{\nu} & (\mathbb{P}^1,B_2),
}
\eeq
where $\varphi$  is a $\mu_d$-equivariant isomorphism, and $\nu$ is an automorphism of $\mathbb{P}^1$ that maps the $i$-th branch point of $\hh_1$ to the $i$-th branch point of $\hh_2$ for each $1 \leq i \leq n$.

An {\it automorphism} of a $\mu_d$-cover $\hh$ is the datum of a commutative diagram:
\beq
\xymatrix{X{\ar@(ul,ur)^{\tau}}   \ar[dr]_{\hh} \ar[rr]^\varphi & & X{\ar@(ul,ur)^{\tau}}   \ar[dl]^{\hh} \\
& \mathbb{P}^1 &
}
\eeq
where $\varphi$ is a $\mu_d$-equivariant automorphism of $X$.

Let $\mathrm{Aut}(X, \tau)$ denote the normalizer of $\tau$ in $\mathrm{Aut}(X)$.
The structure and size of $\mathrm{Aut}(X, \tau)$ do not depend on the choice of $\tau$.
In many cases, $\mathrm{Aut}(X, \tau) = \langle \tau \rangle$.
Note that $\mathrm{Aut}(X, \tau)$ is
the automorphism group of the $\mu_d$-cover determined by $(X,\tau, B)$.

\subsubsection{The genus and signature}

Suppose $\hh: X \to {\mathbb P}^1$ is a $\mu_d$-cover with inertia type $a$.
Then $X$ has an affine equation of the form \eqref{Ecurvegenerala}.
The genus $g$ of $X$ is the dimension of $H^0(X, \Omega^1)$. 
For $0 \leq j \leq d-1$, let $L_j$ be the $j$-th eigenspace for the action of 
$\mu_d$ on $H^0(X, \Omega^1)$  given by precomposing a holomorphic $1$-form with the automorphism of $X$ 
corresponding to a given element of $\mu_d$; we denote $f_j = {\rm dim}(L_j)$.
Then $f_0 =0$. 
The data of $\vec{f} = (f_1, \ldots, f_{d-1})$ is the \emph{signature} of $\hh$.

\begin{lemma} \label{LRH}
Suppose $\hh: X \to {\mathbb P}^1$ is a $\mu_d$-cover with inertia type $a$.
\begin{enumerate}
\item (Riemann-Hurwitz formula) The genus $g$ of $X$ is
$g =  1-d +\frac{1}{2} \sum_{i=1}^n (d- \mathrm{gcd}(d, a_i))$.
\item See, e.g.\ \cite[Lemma~4.3]{Bouw}. 
For $x \in \QQ$, let $\langle x\rangle$ be its fractional part. 
If $1 \leq j \leq d-1$, then
\[ f_j = -1+\sum_{i=1}^n\langle\frac{-ja_i}{d}\rangle.\]
\end{enumerate}
\end{lemma}

If all entries of $a$ are relatively prime to $d$, then $g = (n-2)(d-1)/2$.

\subsubsection{Generically ordinary}

The Frobenius map acts on the set of eigenspaces $\{L_j \mid 1 \leq j \leq d-1\}$ 
via the multiplication-by-$p$ map on the indices.

\begin{proposition} \label{Pbouw} (Special case of \cite[Theorem~6.1]{Bouw})
Let $X_t: y^d = x^{a_1}(x-1)^{a_2}(x-t)^{a_3}$ be a family of $\mu_d$-covers of $\mathbb{P}^1$ branched at $4$ points.
Then the generic curve $X$ in the family is ordinary
if and only if the dimension $f_j$ is constant for each orbit of $\{L_j \mid 1 \leq j \leq d-1\}$ under Frobenius. 
\end{proposition}

\begin{remark} \label{Rcongord}
If $p \equiv 1 \bmod d$, then $X_t$ is ordinary for a generic choice of $t$.  This is because
the orbits of Frobenius on $\{L_j \mid 1 \leq j \leq d-1\}$ all have cardinality $1$,
so the condition that the dimension $f_j$ of $L_j$ is constant within each orbit is vacuous.
\end{remark}

\subsection{Moduli spaces of admissible covers}
\label{sec:msac}

Let $\mathcal{M}_g$ (resp.\ $\overline{\mathcal{M}}_g$) denote the moduli space of smooth (resp.\ stable) curves of genus $g$.
Let $\mathcal{M}_{0,n}$ (resp.\ $\overline{\mathcal{M}}_{0,n}$) denote the moduli space of smooth (resp.\ stable) curves of 
genus $0$ with $n$ marked points.

Let $\mathcal{A}_g$ denote the moduli space of principally polarized abelian varieties of dimension $g$.
Let $\tilde{\mathcal{A}}_g$ denote a smooth toroidal compactification of $\mathcal{A}_g$. 

The Torelli morphism $T: \mathcal{M}_g \to \mathcal{A}_g$ sends a curve of genus $g$ to 
its Jacobian; it is an embedding.  The Torelli morphism can be extended to 
$T: \overline{\mathcal{M}}_g \to \tilde{\mathcal{A}}_g$. 

Let ${\mathcal A}_{d, a}$ denote the moduli space of admissible $\mu_d$-covers of a rational curve $R$
branched at $n$ labeled
points and having inertia type $a$; (this space appeared first in \cite[page 57]{harrismumford};  \cite[Definition 2.7]{COSlambda1} gives a definition adapted to the case we are studying). It is a compactification of the moduli space for $\mu_d$-covers of $\mathbb{P}^{1}$.
Roughly speaking, the boundary points parametrize projection maps obtained from automorphisms of order $d$ on nodal curves whose quotient curves are nodal rational curves.

\begin{remark} 
In fact, ${\mathcal A}_{d, a}$ is the stack of twisted stable maps from an orbifold rational curve to $B\mu_d$ with inertia at the orbifold points given by $a$, as defined in \cite{acv:ac}.  The authors show that this stack is smooth and it is the normalization of the admissible cover space of Harris and Mumford \cite{harrismumford}. The distinction between Harris-Mumford admissible covers and twisted stable maps  is technical and will not appear in this paper, but it is important for the intersection theory to run smoothly.  In particular, a reader familiar only with the Harris-Mumford perspective can follow the arguments, but will have to accept the presence of automorphism related factors in the intersection theory).  
With the twisted stable maps perspective, the signature can also 
be determined with an orbifold Riemann-Roch computation, see \cite[Theorem~7.2.1]{grabervistoli}.
\end{remark}

There is a morphism 
\begin{equation} \label{Edefpi}
\pi=\pi_{d,a}: {\mathcal A}_{d, a} \to \overline{M}_{0,n},
\end{equation}
where the isomorphism class of a $\mu_d$-cover $\hh: X \to R$ is sent to the labeled set of $n$ branch points.
The map $\pi$ is a bijection on closed points, but ${\mathcal A}_{d, a}$ is a $\mu_d$-gerbe over $\overline{M}_{0,n}$. 
This roughly means that it has generic isotropy $\mu_d$. 

There is a morphism
\begin{equation} \label{Edefphi}
\phi=\phi_{d,a}: {\mathcal A}_{d, a} \to \Mgbar,
\end{equation}
where the isomorphism class of a $\mu_d$-cover $\hh: X \to R$ is sent to the isomorphism class of $X$.

\begin{definition} \label{def:od}
A {\it one dimensional family} $\cF$ of admissible $\mu_d$-covers of a rational curve with inertia type $a$ is a morphism $\varphi_{\mathcal{F}}: \cB \to \mathcal{A}_{d,a}$, where $\cB$ is a smooth, proper, irreducible, one dimensional Deligne--Mumford (DM) stack. 
\end{definition}

Such a family $\mathcal{F}$ induces a map 
$\phi^\mathcal{F} = \phi_{d,a} \circ \varphi_{\mathcal{F}}: \cB \to {\overline{\mathcal M}}_g$. 

\begin{definition}[Notation]
We denote by ${\mathcal M}_{d,a}$ (resp.\ $\mathcal{M}^{\mathcal{F}}$)  the image 
under $\phi_{d,a}$ (resp.\ $\phi^\mathcal{F}$) of ${\mathcal A}_{d, a}$ (resp.\ $\cB$)
in ${\overline{\mathcal M}}_g$ with its reduced induced structure.
We denote by $\delta_{d,a}$ (resp.\ $\delta^{\mathcal{F}}$) the degree of $\phi_{d,a}$ 
(resp.\ $\phi^\mathcal{F}$) onto its image.
\end{definition}

\begin{lemma}\label{lem:deg}
Suppose $g \geq 2$ and $n \geq 4$.
The degree $\delta_{d,a}$ of the map $\phi_{d,a}: \mathcal{A}_{d,a} \to \mathcal{M}_{d,a}$ equals the cardinality of the 
set $S_{d,a} := \{( \ell, \sigma)\in (\ZZ/d \ZZ)^\ast \times S_n \mid  \sigma(a) = \ell^{-1} a\}$.
\end{lemma}

\begin{proof}
Let $X$ be a generic point in the image of $\phi_{d,a}$.
Let $\hh: X \to R$ be an admissible $\mu_d$-cover in $\phi_{d,a}^{-1}(X)$.
Let $B$ (resp.\ $\bar{B}$) be the set of labeled (resp.\ unlabeled) branch points of $\hh$.
Then $B$ is a generic point of $\overline{M}_{0,n}$.
Since $n \geq 4$, the only automorphism of $\mathbb{P}^1$ that fixes $B$ is the identity. 
So the other inverse images of $X$ under $\phi_{d,a}$ correspond to the same map $\hh: X\to R$, 
but with possibly different labelings of $\bar{B}$ and different inclusions $\mu_d \subset \mathrm{Aut}(X)$. 
The set of options compatible with the inertia type $a$ is indexed by $S_{d,a}$. 
Here, $\ell \in \mathrm{Aut}(\mu_d) \simeq (\ZZ/d \ZZ)^*$
corresponds to the power of the automorphism $\tau$.
Also, $\sigma$ gives a permutation of the labelings of $\bar{B}$ 
that satisfy the constraint that the $\mu_d$-cover generated by $\tau^\ell$ has inertia type $a$; 
so $\sigma(a) = \ell^{-1} a$ by Remark~\ref{Rswitch}.
\end{proof}

\begin{notation}
In numerous places in this paper, including \eqref{Legendre}, \eqref{E111}, and \eqref{Ehyp},
we describe $\cF$ using an equation of the form $y^d= f(x)(x-t)$.
Here $x$ is the parameter on $\mathbb{P}^1$ and $t$ is the parameter of the family.
Implicitly, the base of this family is $\mathbb{P}^1$: the fibers over the values of $t$ which coincide with a root of $f(x)$ 
are the admissible $\mu_d$-covers in the compactification of the family.
These are obtained by taking the closure of the image via the moduli map of the punctured base
in the space of admissible covers.
\end{notation}

\subsection{The Hodge bundle and Chern classes}

We refer to \cite[Chapters~6-9]{fulton:it} for background on the Chow ring $A^*(\tilde{\cA}_g)$.
Loosely speaking, the Chow ring is generated by equivalence classes of closed subspaces 
under an equivalence relation called rational equivalence. 
We denote the class of a subspace $S$ by $[S]$.
The Chow ring is naturally graded by codimension, denoted in the upper index. 
When only the additive structure is concerned, it is sometime useful to introduce a grading by dimension; in this case a graded piece is given an appropriate lower index.

Recall the morphisms $\pi$ from \eqref{Edefpi} and $\phi$ from \eqref{Edefphi}.
Since $\pi$ has generic isotropy $\mu_d$, 
the degree of $\pi$ is $1/d$, in the sense that pushing forward the fundamental class produces a factor of $1/d$:
$\pi_\ast([{\mathcal A}_{d, a}]) = [\overline{M}_{0,n}]/d$. 
Similarly, $\phi^{\mathcal{F}}_{\ast}([\cB])  = \delta^{\mathcal{F}}[\mathcal{M}^{\mathcal{F}}]$.

The {\it Hodge bundle} $\mathbb{E}_g\to \cA_g$ is a locally free sheaf of rank $g$ on $\cA_g$.
For $A/S$ an abelian scheme, the sections of $\mathbb{E}_g$ are given by 
$e^* \Omega^1_{A/S}$ on $S$, where $\Omega^1_{A/S}$ is the
relative sheaf of differentials and $e: S \to A$ is the identity section.
By \cite{faltingschai}, the Hodge bundle extends to a locally free sheaf on $\tilde{\cA}_g$.

The Chern classes of $\mathbb{E}_g$ are defined over $\mathbb{Z}$ and can be studied over any field.  
They yield classes in $A^*(\tilde{\cA}_g)$. 
We refer to \cite[Chapter~3]{fulton:it} for background on the theory of Chern classes of a vector bundle.
For example, the first Chern class commutes with pullback
and the first Chern class of a line bundle may be represented by the divisor of one of its meromorphic sections.

We are especially interested in the first Chern class:
 \begin{equation}
 \lambda_1:= c_1(\mathbb{E}_g) \in A^*(\tilde{\cA}_g). 
\end{equation}
The class $\lambda_1$ can be represented by the codimension one cycle given 
by the divisor of a meromorphic section of $\det(\Hod)$. 

The Hodge bundle on $\overline{\mathcal{M}}_g$ (resp.\  ${\mathcal A}_{d, a}$) is the pullback of $\Hod$ 
by the Torelli morphism (resp.\ and by $\phi_{d,a}$). 
Over $\mathcal{M}_g$, the fiber of $\mathbb{E}_g$ over a moduli point $X$ is naturally identified with $H^0(X, \Omega^1)$. 
A reference for the extension to the boundary $ \overline{\mathcal{M}}_g\smallsetminus {\mathcal{M}}_g$ is \cite{mumford}.
On these moduli spaces, we continue to denote the Hodge bundle by $\Hod$ 
and its first Chern class by $\lambda_1$.

\begin{definition} \label{def:la}
Let $\mathcal{F}$ be a one dimensional family of admissible $\mu_d$-covers of a rational curve.
Then
\[
{\mathrm{deg}}^{\cF}(\lambda_1):= \int_\cB (\phi^\mathcal{F})^\ast (\lambda_1).
\]
\end{definition}

\subsection{Cycle class of the non-ordinary locus}

Let $V_{g-1}$ denote the locus on ${\mathcal A}_g$ of non-ordinary abelian varieties.
This is the same as the locus of abelian varieties with positive $a$-number.

By \cite[Section 9, page 625]{EVdG}, 
the cycle class of $V_{g-1}$ is $[V_{g-1}] = (p-1) \lambda_1$. We briefly summarize the argument.
The locus $V_{g-1}$ is given by the vanishing of the map 
\[\mathrm{det}(V ) : \mathrm{det}(\mathbb{E}_g ) \to
\mathrm{det}(\mathbb{E}_g^{(p)}).\] 
(Here $\mathbb{E}_g^{(p)}$ denotes the Hodge bundle on the pullback of $\mathcal{A}_g$ by the absolute Frobenius
of $\bar{\FF}_p$.)
By tensoring with $(\mathbb{E}_g)^\vee$, 
this locus is given by the zero locus of the
section  
\[s_V: \mathcal{O}_{\mathcal{A}_g} \to \det(\mathbb{E}_g)^\vee \otimes \mathrm{det}(\mathbb{E}_g^{(p)}).\] 
So $[V_{g-1}] = c_1( \det(\mathbb{E}_g)^\vee \otimes \mathrm{det}(\mathbb{E}_g^{(p)}))$. 
By the definition of $\mathbb{E}_g^{(p)}$, it follows that $c_1(\mathbb{E}_g^{(p)}) = p\lambda_1$.
The first Chern class is additive for line bundles and its sign is reversed by dualization.
So $[V_{g-1}] = -\lambda_1 + p \lambda_1 =(p-1)\lambda_1$.

\section{Main Theorems} \label{SproofmaintheoremF}

In this section, we prove the main theorem stated in the introduction. The two equalities composing its statement 
are proven individually in Theorems~\ref{thm:secondeq} and \ref{thm:firsteq}.
We first recall our definition of the mass formula. 

\begin{definition}
Let $\mathcal{F}$ be a one dimensional family of admissible $\mu_d$-covers of a rational curve with inertia type $a$.  
We define the {\it mass formula} to be
\begin{equation} \label{def:massF}\mu(\cF,p) =\int_{\Mgbar}[{\mathcal M}^{\cF}] \cdot [V_{g-1}] .\end{equation}
\end{definition}

We assume that $p$ is such that 
the source curve of the generic point of the family ${\mathcal{F}} = {\mathcal{F}}_p$ is ordinary.
This implies that $V_{g-1}$ is dimensionally transversal to ${\mathcal M}^{\mathcal{F}}$.  (Two subspaces of a given ambient space are said to be dimensionally transversal if the codimension of their intersection equals the sum of their codimensions.) 
In this situation, the mass formula is the degree of a zero dimensional cycle supported on the set of 
non-ordinary curves of $\cF$. 


\subsection{First main result}

We evaluate the mass formula using tautological intersection theory.

\begin{theorem} \label{thm:secondeq}
For $\cF$ a one dimensional family of admissible $\mu_d$-covers of a rational curve with inertia type $a$, 
and $p \nmid d$ a prime such that the generic curve in the characteristic $p$ fiber $\cF_p$ is ordinary, 
we have:
\begin{equation} \label{EintroFrepeat}
\mu(\mathcal{F},p) 
= (p-1) {\mathrm{deg}}^{\cF}(\lambda_1)/\delta^{\mathcal{F}},
\end{equation}
where the symbols on the right hand side are as in Definitions \ref{def:od}, \ref{def:la}.
\end{theorem}

\begin{proof}
By \eqref{def:massF}, $\mu(\cF,p) = \int_{\Mgbar}[{\mathcal M}^{\cF}] \cdot [V_{g-1}]$.
The intersection $[\mathcal{M}^{\mathcal{F}}]\cdot [V_{g-1}]$ determines a class in the Chow ring of $\Mgbar$. We compute its degree. 

By \cite[Theorem~2.3]{fabervandergeer}, every component of $V_{g-1}$ has dimension $3g-4$.
Furthermore,  the generic point of each component of $V_{g-1}$ represents a smooth curve
by \cite[Lemma~3.2]{AP:mono}.  
By \cite[Section 9, page 625]{EVdG}, the class of the non-ordinary locus $[V_{g-1}]$ on $\mathcal{A}_g$ is $(p-1)\lambda_1$.
So the class $[V_{g-1}]$ on $\Mgbar$ equals $(p-1)\lambda_1\in A^1(\Mgbar)$.

Since $\phi^{\mathcal{F}}_{\ast}([\cB])  = \delta^{\mathcal{F}}[\mathcal{M}^{\mathcal{F}}]$,
 in the Chow ring of $\Mgbar$, we have:
 \beq \label{eq:cmul}
[\mathcal{M}^{\mathcal{F}}]\cdot [V_{g-1}] = \frac{\phi_{\ast}^{\mathcal{F}}([\cB])}{ \delta^{\mathcal{F}}} \cdot (p-1)\lambda_1.
 \eeq
 It thus suffices to show that the degree of $\phi_{\ast}^{\mathcal{F}}([\cB]) \cdot \lambda_1$ equals 
 ${\mathrm{deg}}^{\cF}(\lambda_1)$.
In order to compute the degree of this class,
consider the commutative  diagram 
$$
\xymatrix{
\mathcal{A}^{\mathcal{F}} \ar[rr]^{\phi^{\mathcal{F}}} \ar[dr]_{\tilde{c}}& & \overline{\mathcal{M}}_g \ar[dl]^{c}\\
 & pt. &
}
$$
where $\tilde{c}$ and $c$ are constant functions.      
In the next formula \eqref{eq:okstepone}, the $\lambda_1$ in the left entry denotes the class on 
$\overline{\mathcal{M}}_g$, and the other two denote its pullback to $\mathcal{A}^{\mathcal{F}}$.
Using the projection formula for the morphism $\phi^{\mathcal{F}}$, one obtains

\begin{equation}\label{eq:okstepone}
\phi_{\ast}^{\mathcal{F}}([\cB])\cdot \lambda_1
=  {\phi_{\ast}^{\mathcal{F}}} \left( [\cB] \cdot \lambda_1\right) = {\phi_{\ast}^{\mathcal{F}}} \left( \lambda_1\right).
\end{equation}
Pushing forward  \eqref{eq:okstepone} via  $c$:
\begin{eqnarray*}\label{eq:lhsF}
c_\ast(\phi_{\ast}^{\mathcal{F}}([\cB])\cdot \lambda_1) 
& = & c_\ast{\phi_{\ast }^{\mathcal{F}}} \left( \lambda_1\right) \\
& = & \tilde{c}_\ast \left( \lambda_1\right)  
= {\mathrm{deg}}^{\cF}(\lambda_1)[pt.].
\end{eqnarray*}

This completes the proof because the degrees of the classes are the same.
\end{proof}

\begin{remark} \label{Rgeneral}
The definitions and results in this section apply more generally. 
Any proper one dimensional family of curves $\cF$ determines a class $[\mathcal{C}^\cF] \in A_1(\Mgbar)$, 
and one can define a mass formula for $\cF$ as the degree of the  class $[\mathcal{C}^\cF]\cdot [V_{g-1}]$. 
Then it is immediate that Theorem~\ref{thm:secondeq} holds. 
Since in this work, we only explore applications for families of cyclic covers of rational curves, 
we restrict to that case in the statements.
\end{remark}

\subsection{The mass formula and the multiplicity}

We continue to assume that $V_{g-1}$ is dimensionally transversal to ${\mathcal M}^{\mathcal{F}}$.
For a curve $X\in \Mgbar$, we denote by $m_X$ the multiplicity of intersection of the cycles $ {\mathcal M}^{\calF}_{d,a}$ and $V_{g-1}$ at $X$.  Then
\beq \label{eq:mfwm}
\mu(\cF,p) = \sum_{[X] \in {\mathcal M}^{\cF} \cap V_{g-1}} \frac{m_X}{\#\mathrm{Aut}(X)}.
\eeq

Let $\aaa_X$ denote the order of vanishing of a local equation for the determinant of the Cartier operator at $X$.
We determine $m_X$ in terms of $\aaa_X$ and 
the index of the normalizer of $\tau$ in $\mathrm{Aut}(X)$. 

\begin{proposition} \label{PmxF}
Suppose $p \nmid d$ is such that $V_{g-1}$ is dimensionally transversal to ${\mathcal M}^{\mathcal{F}}$.
If $X \in {\mathcal M}^{\mathcal{F}}$, then $m_X = \aaa_X [\mathrm{Aut}(X,\tau):\mathrm{Aut}(X)]$.
\end{proposition}

\begin{proof}
If $X \not \in V_{g-1}$, then $m_X=0$ and $\aaa_X =0$. Suppose $X \in {\mathcal M}^{\mathcal{F}} \cap V_{g-1}$, 
and $h$ is an inverse image of $X$ via $\phi^\cF$.
Consider the following diagram
\[
\xymatrix{  \ar@{|-{>}}[d]  \hh     \ar@{}[r]|-*[@]{\in} & 
  \cB \ar[rr]^{\phi^\cF} \ar[d] &  &\ar[d] \cM^\cF  \ar@{}[r]|-*[@]{\subset} &  \overline{\cM}_{g,n}   \ar@{}[r]|-*[@]{\ni} & X   \ar@{|-{>}}[d]  \\
  \underline{\hh}   \ar@{}[r]|-*[@]{\in}  &  A^{\cF} \ar[rr]^{\underline{\phi}^\cF} &  & M^\cF  \ar@{}[r]|-*[@]{\subset} & \overline{M}_{g,n}  \ar@{}[r]|-*[@]{\ni} & \underline{X}
    }
\]
where the objects in the second row are the coarse moduli spaces of the DM stacks appearing in the first row, and the vertical arrows are the coarse moduli maps. Given any object in the top row, we denote its coarse image by adding an underline.

Since the hypersurface $\underline{V}_{g-1}$ is a Cartier divisor, 
the coarse multiplicity $m_{\underline{X}}$ may be computed in $A^{\cF}$ by pulling back a local equation $f$ of the hypersurface $\underline{V}_{g-1}$, and computing the order of vanishing at $\underline \hh$. 
Here $f$ is a local equation for the determinant of the Cartier operator.
Thus by definition $m_{\underline{X}} = \aaa_X$. 

To lift the coarse multiplicity to obtain the stacky multiplicity $m_X$, we observe that the moduli point $X\in \Mgbar$ is  isomorphic to the global quotient stack $[pt./ \Aut(X)]$. 
Thus the degree of the coarse moduli map restricted to $X$ is $1/ \# \Aut(X)$. 
Similarly $\hh \cong [pt./ \Aut(X, \tau)]$, and the degree of the coarse moduli map restricted to 
$\hh$ equals $1/ \#\Aut(X,\tau)$. We then obtain
\beq  \label{cmfm}
\frac{m_X}{\# \Aut(X)} = \frac{\aaa_X}{ \#\Aut(X,\tau)}, 
\eeq
 from which the statement of the proposition follows.
\end{proof}

\subsection{Second main result}


We evaluate the mass formula using $a$-numbers of curves in the family.

\begin{theorem}
\label{thm:firsteq}
For $\cF$ a one dimensional family of admissible $\mu_d$-covers of a rational curve with inertia type $a$, 
and $p \nmid d$ a prime such that the generic curve in the characteristic $p$ fiber $\cF_p$ is ordinary, we have:
$$
\mu(\cF, p) =  \sum_{[X]} \frac{\aaa_X}{\#{\mathrm{Aut}}(X, \tau)},
$$
where the sum is over the isomorphism classes of non-ordinary curves $X$ in $\cF_p$. 

Furthermore, $\aaa_X \geq a(X)$, with equality when $n=4$ and $p \equiv 1 \bmod d$. 
\end{theorem}

\begin{proof}
The hypothesis on $p$ implies that $V_{g-1}$ is dimensionally transversal to ${\mathcal M}^{\mathcal{F}}$. 
By Proposition~\ref{PmxF}, 
\[\frac{m_X}{\#\mathrm{Aut}(X)}= \frac{\aaa_X[\mathrm{Aut}(X,\tau):\mathrm{Aut}(X)]}{\#\mathrm{Aut}(X)}=
\frac{\aaa_X}{\#\mathrm{Aut}(X,\tau)}.\]
Substituting this in \eqref{eq:mfwm} completes the first statement.

To compare $\aaa_X$ with $a(X)$, 
without loss of generality, we can suppose the matrix for the Cartier operator is in rational canonical form.
A local equation for the determinant is given by the product of the determinants of the blocks in the matrix.
Then $\aaa_X$ is the sum of the orders of vanishing of the determinants of the blocks, 
while the $a$-number is the number of blocks whose determinant vanishes.
This proves that $\aaa_X \geq a(X)$.

Suppose $n=4$ and $p \equiv 1 \bmod d$.  
The values in the inertia type $a$ have the property that $a_1 + a_2 + a_3 + a_4$ is a multiple of $d$.
Without loss of generality, we can suppose that $a_1 + a_2 + a_3 + a_4$ equals $d$ or $2d$.
(Since $0 < a_i < d$, the sum is less than $4d$.
If $a_1+a_2+a_3+ a_4 = 3d$, then by taking $\ell=-1$, we can adjust to the inertia type $a'=(a_1',a_2',a_3',a_4')$
where $a_i' = d-a_i$ for $1 \leq i \leq 4$, and thus $a_1'+a_2'+a_3'+a_4'=d$.)

The hypothesis that $p \equiv 1 \bmod d$ implies that there is a basis for $H^0(X, \Omega^1)$ for which
the Cartier matrix is diagonal.  Its entries are truncated hypergeometric functions and
the zeros of these are simple by \cite[Lemma~3.7]{wewers03} when the sum is $d$, 
and by \cite[Proposition~5.4, Corollary~5.5]{bouw04} when the sum is $2d$. \end{proof}


\section{Evaluating the class for cyclic covers} \label{SCOS}

In this section, we review results from \cite{COSlambda1}.  Let $d \geq 2$ and $n \geq 4$. 
Recall that $a=(a_1, \ldots, a_n)$ is a tuple of integers with $0 < a_i <d$, whose sum is congruent to $0$ modulo $d$.

\subsection{Cyclic covers branched at four points}

A natural class of examples is given by moduli spaces of cyclic covers
that are branched at exactly $n=4$ points.
In this situation,
we lighten the notation by suppressing the superscript $\cF$.
So ${\mathcal A}_{d,a}$ denotes the moduli space of (admissible) 
$\mu_d$-covers of ${\mathbb P}^1$ branched at $n=4$ points with inertia type $a$. 
In this case, \cite{COSlambda1} gives an explicit formula for the degree of $\lambda_1$.

\begin{theorem}  \label{ThmCOSgeneral} \cite[Theorem~1.2]{COSlambda1}
When $n=4$, the degree of $\lambda_1$ on ${\mathcal A}_{d,a}$ is 
\begin{equation} \label{Ehurwitz4}
{\mathrm{deg}}_{d,a}(\lambda_1) = \frac{1}{12 d^2}\left(d^2 - \sum_{i=1}^4 {\rm gcd}^2(a_i, d)  + \sum_{i=1}^3 {\rm gcd}^2(a_i + a_4, d)\right).
\end{equation}
\end{theorem}

This theorem is proven using Atiyah-Bott localization to obtain linear relations between the degree of $\lambda_1$ and the degrees of some zero-dimensional boundary strata in the space of admissible covers. Alternatively, one may deduce this formula from a stacky Grothendieck-Riemann-Roch computation, as done in \cite[Proposition 10.20]{bertinromagny} in the hyperelliptic case.

\subsection{Cyclic covers with more than four branch points}

In \cite{COSlambda1}, the authors also evaluate the class $\lambda_1$
in the higher dimensional case when the covers have more than four branch points;
this generalization is proven inductively with the base case being Theorem~\ref{ThmCOSgeneral}.
We will need this result in Section~\ref{Slinhyp}, in a situation when the base of the family is not all of ${\mathcal A}_{d,a}$.

First we define some notation.
Suppose $J$ is a subset of $[n]=\{1, \ldots, n\}$ with $2\leq |J|\leq n-2$.
Then $\Delta_J$ denotes the boundary divisor whose generic point represents 
a $\mu_d$-cover of a rational curve with two irreducible components, intersecting in one node,
with the branch points labeled by $J$ on one component, and the branch points labeled by $J^c$ on the other. Note that $\Delta_J = \Delta_{J^c}$ and therefore each divisor is counted twice when summing over all allowed subsets of $[n]$.

We need some other definitions.
Consider the universal curve $\pi: \mathcal{C}_{g,n} \to \overline{\mathcal{M}}_{g,n}$ and the 
section $\sigma_i$ of $\pi$ whose image on an $n$-marked curve is the $i$th marked point.
Define $\psi_i = c_1(\sigma_i^* (\omega_\pi))$, where $\omega_\pi$ is the relative dualizing sheaf.

The general point $x$ of $\mathcal{C}_{g,n}$ can naturally be thought of as an $(n+1)$-th mark on the $n$-pointed curve $\pi(x)$.
This gives a birational morphism $\mathcal{C}_{g,n} \to \overline{\mathcal{M}}_{g,n+1}$ which extends to an isomorphism. 
Thus we can view $\pi = \pi_{n+1}$
as the projection map that forgets the data of the last marked point.
Define $\kappa_1 = \pi_{n+1, *} (\psi_{n+1}^2)$.

\begin{theorem} \cite[Theorem~1.3]{COSlambda1} \label{thm:gf} 
The class $\lambda_1$ on ${\mathcal A}_{d,a}$ admits the following tautological representation:
\begin{equation} \label{eq:gf}
 \lambda_1 = \frac{1}{24 d} \left( \sum_{J\in \mathcal{P}([n])} {\gcd}^2\left(\sum_{j\in J}a_i, d\right) \Delta_J \right),   
\end{equation}
 where the sum runs over  subsets $J$ of $\{1, \ldots, n\}$, and
\begin{itemize}
\item if $2\leq |J|\leq n-2$, then $\Delta_J$ denotes the boundary divisor described above;
\item if $J = \{j\}$ or $J=[n]\smallsetminus \{j\}$ for some $1 \leq j \leq n$, then $\Delta_J:= -\psi_j$;
\item if $J  = \phi$ or $J=[n]$, then $\Delta_J:= \kappa_1$.
\end{itemize}
\end{theorem}


\section{Linearized families of hyperelliptic curves of every genus} \label{Slinhyp}

In Section~\ref{Selliptic}, we show that our mass formula generalizes the Eichler--Deuring mass formula.
Then, we determine the mass formula for a one dimensional family of hyperelliptic curves of every genus $g \geq 2$. 

Here is a fact used in this section, and in Sections~\ref{ScaseB} and \ref{Sanothercase}: 
if $X$ is a hyperelliptic curve, with hyperelliptic involution denoted $\tau$, then $\tau$ is in the center of 
$\mathrm{Aut}(X)$, so $\mathrm{Aut}(X, \tau) = \mathrm{Aut}(X)$; in this situation, the 
reduced automorphism group is $\mathrm{redAut}(X) := \mathrm{Aut}(X)/\langle \tau \rangle$.

\subsection{The Legendre family as an example} \label{Selliptic}

Suppose $d=2$, $n=4$, and $a=(1,1,1,1)$.  Then the family \eqref{Ecurvegenerala} 
specializes to the Legendre family of elliptic curves $E_t:y^2=x(x-1)(x-t)$ from \eqref{Legendre}.

\begin{lemma} \label{Llegendre}
In the case of the Legendre family \eqref{Legendre}, 
Theorems~\ref{thm:secondeq} and \ref{thm:firsteq} specialize to the Eichler--Deuring mass formula \eqref{EED}.
\end{lemma}

\begin{proof}
By Theorem~\ref{ThmCOSgeneral}, ${\mathrm{deg}}^{\cF}(\lambda_1)=1/4$; unfortunately, the proof 
of that result does not simplify in this case.
We compute $\delta_{2, (1,1,1,1)}=6$: one of the ramification points is the marked point 
on the elliptic curve and there are $3!$ labelings of the other three ramification points.
By Theorem~\ref{thm:firsteq}, $\mu(\cF, p) = (p-1)/24$.

In the Legendre family \eqref{Legendre}, $\tau$ is the hyperelliptic involution, 
so ${\mathrm{Aut}}(E_t, \tau) = {\mathrm{Aut}}(E_t)$.
Also, if $E_t$ is non-ordinary, its $a$-number is $a(E_t)=1$.  
By Theorem~\ref{thm:firsteq}, $\mu(\cF,p) = \sum_{[E]} \frac{1}{\#{\rm Aut}(E)}$,
where the sum is over the isomorphism classes of non-ordinary elliptic curves $E$ over $k$.
By equating these two expressions for $\mu(\cF,p)$, we obtain \eqref{EED}.
\end{proof}

There are many proofs of the Eichler--Deuring mass formula; our proof of the Main Theorem is most similar
to the one in \cite[Corollary~12.4.6, page 358]{katzmazur}.
Other proofs can be deduced from: 
the separability of the Deuring polynomial \cite[Theorem~4.1, chapter~13]{husemoller};
 a comparison of the $\ell$-adic \'etale Euler characteristic of the modular curve
$Y_0(p)$ in characteristic $0$ and characteristic $p$; or
a computation of the constant term of the weight two Eisenstein series on $\Gamma_0(p)$.

\begin{remark} \label{Rlegendre}
The number of isomorphism classes of supersingular elliptic curves 
determines (and is determined by) \eqref{EED}, given
some additional information about elliptic curves with extra automorphisms, 
as seen in \cite[Section 13.4]{husemoller}.
This idea shows up again for curves of genus $2$ in Sections~\ref{SikoA} and \ref{SikoB}.
It is no longer easy to switch between masses and cardinalities when $g >2$.
\end{remark}


 

\subsection{The mass formula for linearized families of hyperelliptic curves} \label{Smasshyp}

\begin{definition} \label{Dlinear}
Suppose $h(x) \in k[x]$ is a separable polynomial of degree $2g+1$.
A {\it linearized family of hyperelliptic curves} is a   one dimensional family $\cF_{h(x)}$ of hyperelliptic curves of genus $g$ whose generic fiber is given by the affine equation
\begin{equation} \label{Ehyp}
X_t: y^2 = h(x)(x-t).
\end{equation} 
\end{definition}

We assume that the roots of $h(x)$ do not satisfy any $\mathrm{PGL}_2(k)$ symmetry. 
We also assume that the generic curve in the characteristic $p$ fiber of $\cF_{h(x)}$ is ordinary. 

Here is our first corollary of the main theorem.

\begin{corollary} \label{Chyp}
Let $g \geq 2$ and $p$ be odd.  
Suppose $h(x) \in k[x]$ is a separable polynomial of degree $2g+1$ whose set of roots is not stabilized by any 
$\mathrm{PGL}_2(k)$ symmetry.
Suppose $p$ is a prime such that the generic curve in the characteristic $p$ fiber of 
$\cF_{h(x)} : y^2 = h(x)(x-t)$ is ordinary. 
Then the mass formula is $\mu(\cF_{h(x)}, p) = (p-1)g/4$.
\end{corollary}

The proof of Corollary~\ref{Chyp} is in Section~\ref{Sproof}.

\begin{remark}
In \cite[page 381]{Yui}, Yui provides a method to compute the Cartier operator for the curve $X_t: y^2 = h(x)(x-t)$.
One can use this method to show that there are at most $g(p-1)/2$ values of $t$ for which 
$X_t$ is non-ordinary in the family \eqref{Ehyp}.
For a generic non-ordinary hyperelliptic curve $X$ of genus $g \geq 2$, it is known that:
$\#\mathrm{Aut}(X) = 2$ by \cite[Theorem~1.1(ii)]{AchterGlassPries}; and the curve $X$ has $p$-rank $g-1$ and 
and $a_X = 1$ by \cite[Theorem 1]{GlassPries05}.
However, this information is not sufficient to prove Corollary~\ref{Chyp} because it is not clear whether 
these facts are true in the family $\cF_{h(x)}$, nor do they give information about $\alpha_X$.
\end{remark}

\subsection{Proof of mass formula for linearized hyperelliptic families} \label{Sproof}

In the proof of Corollary~\ref{Chyp}, we use the following facts about 
tautological classes on admissible covers and $\overline{M}_{0,n}$.

\begin{lemma} \label{factoids}
Let $g \geq 2$ and $d=2$.  Let $a=(1, \ldots, 1)$ be a tuple of length $n=2g+2$.
Consider the space of hyperelliptic admissible covers $\cA_{2,a}$ and the tautological morphism $\phi_{2,a}: \cA_{2,a} \to \overline{M}_{0,2g+2}$ of degree $1/2$.
Then:
\begin{enumerate}
\item On the universal family of hyperelliptic covers, 
consider a (marked) ramification point $r$ and its corresponding branch point $b$. 
The class $\psi_{r}$ is restricted from $\overline{M}_{g,2g+2}$, 
whereas the class $\psi_{b}$ is in $A^1(\overline{M}_{0,2g+2})$. We have:
$\psi_{r} =\phi_{2,a}^\ast \psi_{b}/2$.
\item Suppose $i \in [n]$.  Choose two distinct values $j,k \in [n] -\{i\}$. 
Then the class $\psi_{i}\in A^1(\overline{M}_{0,n})$ is represented by the following linear combination of boundary divisors:
\[
\psi_i = \sum_{\begin{array}{c} I \subset [n] \mbox{ such that } i\in I\\   j,k\not\in I\end{array}} \Delta_I.
\]
\item The following relation holds in $A^1(\overline{M}_{0,n})$:
\[
\kappa_1 = \sum_{1 \leq i \leq n} \psi_i  - \Delta_{tot},
\]
where $\Delta_{tot} = \frac{1}{2} \sum_{2 \leq |J| \leq n-2} \Delta_J$ denotes the sum of all boundary divisors.
\end{enumerate}

\end{lemma}

\begin{proof}
The first statement is the specialization to the degree $2$  setting of \cite[Lemma 1.17]{ionel}. The second statement follows from the initial condition $\psi_i = 0$ on $\overline{M}_{0,\{i,j,k\}}$ and the relation for $\psi$ classes when pulled-back via forgetful morphisms \cite[Lemma 1.3.1]{k:pc}. The third statement follows from the initial condition $\kappa_1 = 0$ on $\overline{M}_{0,3}$ and the relation for $\kappa$ classes when pulled-back via forgetful morphisms \cite[Lemma 2.2.3]{k:pc}. 
The last two statements can be found in a different form in \cite[Formulas (1.9), (1.10)]{ac:comb}.
\end{proof}

\begin{proof} [Proof of Corollary~\ref{Chyp}]
Write $\cF=\cF_{h(x)}$.
By Theorem~\ref{thm:secondeq}, 
$\mu(\cF,p) 
= (p-1) {\mathrm{deg}}^{\cF}(\lambda_1)/\delta^{\cF}$.
By hypothesis, there is no automorphism of $\bP^1$ that stabilizes the set of roots of $h(x)$; 
it follows that $\delta^{\cF} = 1$.
It thus suffices to prove that ${\mathrm{deg}}^{\cF}(\lambda_1) = g/4$.

We compute the intersection numbers of the family $\cF$ with the tautological divisors appearing on the right hand side of \eqref{eq:gf}. Let $b_i$ for $1 \leq i \leq 2g+1$ denote the roots of $h(x)$.
Then: 
\begin{enumerate}

\item If $2 \leq |J| \leq 2g$, then $\cF \cdot \Delta_J=0$ unless $J$ or $J^c$ is $\{i, 2g+2\}$ for some $1 \leq i \leq 2g+1$.
The reason is that the only branch point that moves in $\cF$ is the last one.
So the only time $\cF$ hits a boundary divisor is when the last branch point specializes to one of the others.

\item If $J =  \{i, 2g+2\}$ for some $1 \leq i \leq 2g+1$, then $\cF \cdot \Delta_J = \frac{1}{2} [pt]$. 
This intersection occurs when $t$ specializes to the branch point $b_i$.  The singular curve
has two components, with genera $0$ and $g-1$, intersecting in two ordinary double points; 
on each component, these two points are an orbit for the restriction of the hyperelliptic involution.
The coefficient $1/2 = 2/(2\cdot 2)$ accounts for the hyperelliptic involution on each component, and the two choices of ways to identify the two points on the two components.

\item If $i \not = 2g+2$, then $\cF \cdot \psi_{i} = \frac{1}{2} [pt]$.  To see this, 
we first use Lemma~\ref{factoids}, (1) to work in $A^1(\overline{M}_{0,2g+2})$.
We use Lemma \ref{factoids}(2) to replace $\psi_i$ by a linear combination of boundary divisors. 
We choose $j=2g+2$ and choose $k$ arbitrarily. 
Then the only boundary divisor $\Delta_I$ intersecting $\cF$ is when $I = [n] - \{k, 2g+2\}$.
This is because the branch points labeled by $2g+2$ and $k$ must be on the second component.
Since the branch point labeled with $2g+2$ is $t$ and $t$ can only specialize to one other branch point, 
all the other branch points must be on the first component.
The result then follows from part (2).

\item We compute $\cF \cdot \psi_{2g+2} = \frac{2g-1}{2} [pt]$. 
The proof begins in the same way as for part (3).
In this case, we choose $j=1$ and $k=2$. 
If a boundary divisor $\Delta_I$ in the expression of $\psi_{2g+2}$ intersects $\cF$, then the marked branch point $t$ must be on a component alone with only one other marked branch point $b_i$, which must be different from $b_1$ and $b_2$. 
There are $2g-1$ possible choices for $i \in \{3, \ldots 2g+1\}$ and therefore $2g-1$ divisors $\Delta_I$ intersecting $\cF$, 
with $I=\{i,2g+2\}$.  The result then follows from part (2).

\item We compute $\cF \cdot \kappa_1 = \frac{2g-1}{2} [pt]$. 
To see this, note that $\kappa_1 =  \psi_{2g+2} + \sum_{i = 1}^{2g+1} \psi_i  - \Delta_{tot}$ by Lemma \ref{factoids}(3).
By part (3), $\cF \cdot  \sum_{i = 1}^{2g+1} \psi_i  =  \frac{2g+1}{2} [pt]$. 
By part (1), $\cF \cdot \Delta_{tot} = \sum_{1 \leq i \leq 2g+1} \cF \cdot \Delta_{\{i, 2g+2\}}$, which equals $\frac{2g-1}{2} [pt]$ by part (2).  By part (4), $\cF \cdot  \psi_{2g+2}=\frac{2g-1}{2} [pt]$.
\end{enumerate}

Finally, we substitute these intersection numbers in \eqref{eq:gf}.
The coefficient $\mathrm{gcd}^2(\sum_{i \in J} a_i, d)$ equals $1$ if $|J|$ is odd and equals $4$ if $|J|$ is even.
For $1 \leq s \leq 2g+2$, let $C_s=\sum_{|J| = s} \cF \cdot \Delta_J$.
By definition, $C_s = C_{2g+2-s}$.
By part (1), $C_s =0$ for $3 \leq s \leq 2g-1$.
So \[\int_{\cF} \lambda_1 =  \frac{1}{24} \cdot \frac{2}{2} (4 C_2 + C_1 + 4 C_0).\]
By part (1), $C_2 = \frac{2g-1}{2}$.
By parts (3)-(4), $C_1 =  (-1) \frac{(2g-1)+(2g+1)}{2}=-2g$.
By part (5), $C_0 = \frac{2g+1}{2}$. 
This yields that $\int_{\cF} \lambda_1 =g/4$.
\end{proof}

\section{Covers branched at four points} \label{Sonedac}

As our second corollary,
we compute the mass formula for any family of cyclic covers of $\mathbb{P}^1$ branched at four points.

\begin{corollary} \label{C4point}
Let $d \geq 2$ and $n=4$.  Let $a=(a_1, \ldots, a_4)$ be an inertia type for $d$.
Let $\cF = \cA_{d,a}$ be the one dimensional family of admissible $\mu_d$-covers of a rational curve with inertia type $a$.
Let $p$ be a prime such that the generic curve in the characteristic $p$ fiber $\cF_p$ is ordinary. 
Then
\begin{equation} \label{EintroFagain}
\mu(\mathcal{F}_{d, a ,p}) 
= (p-1) {\mathrm{deg}}_{d,a}(\lambda_1)/\delta_{d,a},
\end{equation}
where the formula for ${\mathrm{deg}}_{d,a}(\lambda_1)$ is in \eqref{Ehurwitz4}
and the formula for $\delta_{d,a}$ is in Lemma~\ref{lem:deg}.
\end{corollary}

\begin{proof}
Immediate from Theorems~\ref{thm:secondeq} and \ref{ThmCOSgeneral}.
\end{proof}

\begin{example} \label{Ecase111}
Suppose $d \geq 5$ with $\mathrm{gcd}(d,6) = 1$ and $a = (1,1,1,d-3)$.
Suppose $p \equiv 1 \bmod d$.  
Then the mass formula for the number of 
isomorphism classes of non-ordinary curves in the family 
$X_t: y^d = x(x-1)(x-t)$ is 
$\mu(d,a,p) = (p-1) (d^2-1)/(72 \cdot d^2)$.
\end{example}

\begin{proof}
The condition $p \equiv 1 \bmod d$ implies that the source curve of the generic point of $\cA_{d,a}$ is ordinary
by Remark~\ref{Rcongord}. 
The hypotheses on $d$ and $a$ imply that ${\mathrm{deg}}_{d,a}(\lambda_1) = (d^2-1)/12d^2$ by 
Theorem~\ref{ThmCOSgeneral}.
By Lemma~\ref{lem:deg}, $\delta_{d,a} = 6$.
The result follows from Corollary~\ref{C4point}.
\end{proof}



\subsection{Special families} \label{Sspecial}

We describe some families of cyclic covers of $\mathbb{P}^1$ branched at four points 
for which we have complete information about the $a$-numbers.
As a result, we show that the number of isomorphism classes
of non-ordinary curves in these families grows linearly in $p$, and find the asymptotic rate of growth of this number.
There are 10 families in Corollary~\ref{Cmoonen} for which this result is new, specifically those with $g \geq 3$. 

The family ${\mathcal M}_{d,a}$ is called {\it special} if its image under the Torelli morphism is open and 
dense in a component of the associated Deligne--Mostow Shimura variety.  Moonen classified the
families of $\mu_d$-covers of $\mathbb{P}^1$ that are special.  There are 20 such families
(up to relabeling the branch points or changing the generator of $\mu_d$); 
we label these as $M[r]$ for $1 \leq r \leq 20$ as in \cite[Table~1]{moonen}.
Of these, $14$ are one dimensional.  

Here is the key feature we use about each of the $14$ one dimensional special Moonen families: 
As seen in \cite[Section~6]{LMPT2}, for $p \equiv 1 \bmod d$, there is only one option, denoted $a_\nu$,
for the $a$-number of the non-ordinary curves in the family.

\begin{notation}
Suppose $p \equiv 1 \bmod d$.
For each Moonen family $M[r] =\mathcal{M}_{d,a}$ that is one dimensional:
let $\delta_{d,a} = \mathrm{deg}(\mathcal{A}_{d,a} \to \mathcal{M}_{d,a})$;
let $z_{d,a}$ be the size of $\mathrm{Aut}(X,\tau)$ for a generic curve $X$ in $M[r]$; \\
let $a_\nu$ be the $a$-number of the non-ordinary curves on $M[r]$;
and ${\mathrm{deg}}_{d,a}(\lambda_1)$ as in \eqref{Ehurwitz4}.
\end{notation}

\begin{corollary} \label{Cmoonen}
For the one dimensional special Moonen families $\cF=M[r]$:
the table below includes the data of $d$, $a$, $g$, ${\mathrm{deg}}_{d,a}(\lambda_1)$, $\delta_{d,a}$, $z_{d,a}$, and $a_\nu$.  

For $p \equiv 1 \bmod d$: 
the mass formula for the non-ordinary curves in $\cF$ is $\mu(\cF, p) = (p-1)C/\delta_{d,a}$;
the number of isomorphism classes of non-ordinary curves in $\cF$ grows linearly in $p$, 
with asymptotic rate of growth 
$(p-1)n_{d,a}$, where $n_{d,a} = z_{d,a} C/a_\nu \delta_{d,a}$ is given below.

\begin{small}
	\begin{center}
		\begin{tabular}{  |c|c|c|c|c|c|c|c|c|  }
\hline
Label & $d$& $a$ & $g$ & ${\mathrm{deg}}_{d,a}(\lambda_1)$ & $\delta_{d,a}$ & $z_{d,a}$ & $a_\nu$ & $n_{d,a}$ \\
\hline
$M[1]$ & $2$ & $(1,1,1,1)$ & $1$ 
& $1/4$ & $6$ & $2$ & $1$ & $1/12$ \\ \hline
$M[3]$ & $3$ & $(1,1,2,2)$ & $2$ 
& $2/9$ & $8$ & $12$ & $2$  & $1/6$ \\ \hline
$M[4]$ & $4$ & $(1,2,2,3)$ & $2$ 
& $1/8$ & $4$ & $8$ & $2$  & $1/8$ \\ \hline
$M[5]$ & $6$ & $(2,3,3,4)$ & $2$ 
& $1/9$ & $4$ & $12$ & $2$ & $1/6$ \\ \hline	
$M[7]$ & $4$ & $(1,1,1,1)$ & $3$ 
& $1/8$ & $24$ & $16$ & $1$  & $1/12$  \\ \hline
$M[9]$ & $6$ & $(1,3,4,4)$ & $3$ 
& $1/18$ & $2$ & $6$ &  $2$  & $1/12$ \\ \hline
$M[11]$& $5$ & $(1,3,3,3)$ & $4$ 
& $2/25$ & $6$ & $5$ & $2$  & $1/30$  \\ \hline
$M[12]$ & $6$ & $(1,1,1,3)$ & $4$ 
& $1/12$ & $6$ & $6$ &  $1$  &  $1/12$  \\ \hline
$M[13]$ & $6$ & $(1,1,2,2)$ & $4$ 
& $1/9$ & $4$ & $12$ & $2$  & $1/6$  \\ \hline
$M[15]$& $8$ & $(2,4,5,5)$ & $5$ 
& $1/16$ & $2$ & $8$ & $2$  & $1/8$  \\ \hline
$M[17]$ & $7$ & $(2,4,4,4)$ & $6$ 
& $4/49$ & $6$ & $7$ & $2$  & $1/21$   \\ \hline
$M[18]$ & $10$ & $(3,5,6,6)$ & $6$ 
& $3/50$ & $2$ & $10$ &  $2$  & $3/10$  \\ \hline
$M[19]$ & $9$ & $(3,5,5,5)$ & $7$ 
& $2/27$ & $6$ & $9$ & $2$  & $1/18$   \\ \hline
$M[20]$ & $12$ & $(4,6,7,7)$ & $7$ 
& $1/18$ & $2$ & $12$ & $2$  & $1/6$  \\ \hline
		\end{tabular}
	\end{center}
\end{small}
\end{corollary}

\begin{remark}
\begin{enumerate}
\item The family $M[1]$ is the Legendre family.
The family $M[3]$ (resp.\ $M[4]$) is studied in detail in Section~\ref{SikoA}
(resp.\ Section~\ref{SikoB}) with no congruence condition. 
The families $M[5]$ and $M[3]$ have the same image in $\mathcal{M}_2$.
\item
The computation of $n_{d,a}$ is new for the remaining 10 families.

\item Corollary~\ref{Cintro} is immediate for $d=5$ from the data for $M[11]$.
This is because the inertia type $(1,3,3,3)$ for $d=5$ is equivalent to the inertia type $(1,1,1,2)$, which gives the 
equation $y^5=x(x-1)(x-t)$, after changing the $5$th root of unity and relabeling the branch points.  
Similarly, Corollary~\ref{Cintro} is immediate for $d=7$ from the data for $M[17]$.

\item Note that $(p-1) n_{d,a}$ is not always an integer; this is due to the fact that there 
are several exceptional points in the family such that the curve has larger automorphism group.
\end{enumerate}
\end{remark}

\begin{proof}
The value of ${\mathrm{deg}}_{d,a}(\lambda_1)$ comes from Theorem~\ref{ThmCOSgeneral}.
The value of $\delta_{d,a}$ comes from Lemma~\ref{lem:deg}.

Every automorphism in $\Aut(X,\tau)$ descends to an automorphism of $\mathbb{P}^1$ that stabilizes $\{0,1,\infty, t\}$ and is 
compatible with the inertia type $a$.  
We calculate the value of $z_{d,a}$ for $M[3], M[4], M[5]$ in Sections~\ref{SikoA} and \ref{SikoB}.
For the other families, one can compute that
$\Aut(X,\tau)=\langle \tau \rangle$ for a generic value of $t$, except for $M[13]$.

By Theorem~\ref{thm:firsteq}, if $n=4$ and $p \equiv 1 \bmod d$, then $\aaa_X=a(X)$.
Here is the key feature we use about each of the $14$ one dimensional special Moonen families: 
for $p \equiv 1 \bmod d$, there is only one option, namely $a_\nu$,
for the $a$-number at the non-ordinary (basic) points of the family.
Thus, if $X$ is a non-ordinary curve in $\cF$, then $\aaa_X=a_\nu$.

The Newton polygons for the Moonen families are listed in \cite[Section~6]{LMPT2}.  From this, we determine 
the value of $a_\nu$.
By Remark~\ref{Rcongord}, if $p \equiv 1 \bmod d$, then the generic point of $\mathcal{M}_{d,a,p}$ 
represents an ordinary curve, so the hypotheses of Theorems~\ref{thm:firsteq} and \ref{thm:secondeq} are satisfied and the result follows.
\end{proof}

\section{A family of hyperelliptic curves with dihedral action} \label{ScaseB}

In this section, we provide an example of the mass formula for every even genus.  
In Section~\ref{SikoA}, we show that this material
generalizes results when $g=2$ from \cite{IKO} and \cite{hasegawa}.

Throughout the section, suppose $d \geq 3$ is odd and $a=(1,1,d-1,d-1)$. 
Let $p$ be an odd prime with $p \nmid d$.
Recall that $k$ is an algebraically closed field of characteristic $p$.
For $t \in k -\{0,1\}$, we consider the family of curves of genus $d-1$ given by
\begin{equation} \label{EcaseB}
X_t:  y^d=x(x-1)(x-t)^{d-1}.
\end{equation}

\subsection{The mass formula for $a=(1,1,d-1,d-1)$}

\begin{corollary}\label{CcaseB}
Suppose $d \geq 3$ is odd, $a=(1,1,d-1,d-1)$, and $p \nmid 2d$.
The mass formula for the
non-ordinary curves in the family $\cF$ given by $X_t:  y^d=x(x-1)(x-t)^{d-1}$ is 
\[\mu(\cF,p) = (p-1) (d^2-1)/2^5 d^2.\]
\end{corollary}

\begin{proof}
By Lemma~\ref{LRH}, the signature of the $\mu_d$-cover has dimensions $f_j =1$ if $1 \leq j \leq d-1$, because
\[f_j  =  
-1 + 2 \langle \frac{-j}{d} \rangle + 2 \langle \frac{-j(d-1)}{d} \rangle =
-1 + 2 \frac{d-j}{d} + 2 \frac{j}{d} = 1.\]
By Proposition~\ref{Pbouw}, 
$X_t$ is ordinary for the generic value of $t$. 
By Lemma~\ref{lem:deg}, $\delta_{d,a} = 8$.
By Theorem~\ref{ThmCOSgeneral},
${\mathrm{deg}}_{d,a}(\lambda_1) = (d^2 - 4 + 2d^2+1)/12d^2 = (d^2-1)/4d^2$.
We then apply Theorem~\ref{thm:secondeq}.
\end{proof}

In order to use Corollary~\ref{CcaseB}
to estimate the number of non-ordinary curves in the family \eqref{EcaseB}, 
we need more information about the automorphism groups and $a$-numbers.

\subsection{The automorphism group and $a$-number}

\begin{lemma} \label{L11hyp}
Let $d \geq 3$ be odd.  Let $X_t: y^d=x(x-1)(x-t)^{d-1}$.
\begin{enumerate}
\item \cite[Lemma~2.1]{hasegawa}
Then $X_t$ is hyperelliptic.  In fact, $X_t$ is birationally equivalent to
\begin{equation} \label{EhypcaseB}
Y^2=W^{2d}+(2-4t)W^d +1,
\end{equation}
with the hyperelliptic involution $\iota$ given by $\iota(Y,W)=(-Y,W)$.
\item The curves $X_{t_1}$ and $X_{t_2}$ are isomorphic if and only if either $t_2=t_1$ or $t_2=1-t_1$.
\item If $t \not = 1/2$, then $\mathrm{Aut}(X_t, \tau) \simeq C_2 \times D_d$, where 
$D_d$ is the dihedral group of order $2d$.
\end{enumerate}
\end{lemma}

\begin{proof}
\begin{enumerate}
\item We include the proof for convenience.
Write (i) $Z=W^d+1-2t$.
One can check that \eqref{EhypcaseB} is true if and only if
(ii) $(Y+Z)(Y-Z) = 4t(1-t)$.
Let $x=(Y+W^d+1)/2$ and note that $x-t=(Y+Z)/2$.
Then (ii) implies that $(Y-Z)/2 = t(1-t)/(x-t)$.

Write (iii) $Z=(Y+Z)/2 - (Y-Z)/2 = (x-t) + t(t-1)/(x-t)$.
Substituting (iii) for $Z$ in (i) and multiplying by $(x-t)^d$ yields the equation
\[(x-t)^d W^d +(x-t)^d(1-2t) = (x-t)^{d-1}((x-t)^2 + t(t-1)).\]
Let $y=(x-t)W$, then the equation simplifies to 
$y^d = (x-t)^{d-1}x(x-1)$.

\item By part (1), $X_{t_1} \simeq X_{t_2}$ if and only if
$Y^2=f_1(W)$ and $Y^2=f_2(W)$ are isomorphic, where $f_i(X) = W^{2d} + (2-4t_i)W^d +1$.
An isomorphism between hyperelliptic curves
descends to a fractional linear transformation $\gamma$. 
Without loss of generality, we can suppose $\gamma$ fixes $\infty$,
because the map $W \mapsto 1/W$ preserves the roots of $f_1(W)$. 
Thus $X_{t_1} \simeq X_{t_2}$ if and only if there exists a map $\gamma(W)=aW+b$ 
such that $f_1(W)=f_2(\gamma(W))/a^{2d}$.  
This is only possible if $b=0$, $a^{2d}=1$ and $(2-4t_2)/a^d = 2-4t_1$.
If $f_1(W) \not = f_2(W)$, this implies that 
$t_2 = 1-t_1$.   
Conversely, if $t_2 = 1-t_1$, then $\gamma(W)=-W$ provides the isomorphism.

\item 
By part (1), $X_t$ is isomorphic to
$Y^2=f(W)$ where $f(W) = W^{2d}+(2-4t)W^d +1$.
Note that $\langle \iota \rangle \simeq C_2$ and $\iota$ is in the center of $\mathrm{Aut}(X_t)$.
The order $d$ automorphism $\tau(x,y)=(x, \zeta_d y)$ acts by $\tau(Y, W)=(Y, \zeta_d W)$.
The automorphism $\gamma(Y, W) = (Y/W^d, 1/W)$ has order $2$.
A short computation shows that $\gamma \tau \gamma^{-1} = \tau^{-1}$.
Thus $\mathrm{Aut}(X_t, \tau)$ contains a subgroup isomorphic to $C_2 \times D_d$.
Arguments similar to those in part (2) show that these are the only automorphisms that normalize $\tau$ unless $t=1/2$.
\end{enumerate}
\end{proof}

We thank Everett Howe and Jen Paulhus for conversations about the next lemma. 
Following \cite[Section 4]{ries}, we define two curves 
$Z_{+1,t}$ and $Z_{-1,t}$ of genus $(d-1)/2$ that are quotients of $X_t$.
For a positive integer $n$, let $P_n(S) \in\ZZ[S]$ be such that $P_n(S+S^{-1}) = S^n + S^{-n}$.
Then $P_0(S)=2$, $P_1(S)=S$, $P_2(S)=S^2-2$.
These satisfy the recurrence relation $P_{n+2}(S) = S \cdot P_{n+1}(S) - P_n(S)$.
The function $P_n(S)$ is odd (resp.\ even) when $n$ is odd (resp.\ even).

For $\epsilon = 1, -1$, let $Z_{\epsilon,t}$ be the hyperelliptic curve $v^2 = (u + 2 \epsilon) P_d(u) + (2 - 4t)$. 

\begin{lemma} \label{L11hypaa}
Consider $X_t: y^d=x(x-1)(x-t)^{d-1}$.
\begin{enumerate}
\item \cite[Theorem~4.2]{ries}
There is an isomorphism $\mathrm{Jac}(X_t) \simeq \mathrm{Jac}(Z_{+1, t}) \times \mathrm{Jac}(Z_{-1,t})$ of abelian varieties without polarization.
\item The curves $Z_{+1, t}$ and $Z_{-1, t}$ are isomorphic.
\end{enumerate}
\end{lemma}

\begin{proof} 
\begin{enumerate}
\item This is proved in \cite[Theorem~4.2]{ries} over ${\mathbb C}$.  
The proof holds over $k$ because $p \nmid 2d$.

\item The curve $Z_{\epsilon,t}$ has equation $v^2 = m_{\epsilon}(u)$ where 
$m_\epsilon(u) = (u \cdot P_d(u) + 2-4t) + 2 \epsilon P_d(u)$.
Since $u \cdot P_d(u)$ is even and $P_d(u)$ is odd, it follows that $m_{-1}(u) = m_{+1}(-u)$.  So the 
curves $Z_{+1, t}$ and $Z_{-1,t}$ are isomorphic by the change of variables $u \mapsto -u$.
\end{enumerate}
\end{proof}

\begin{lemma} \label{Laaa=a2} 
Suppose $p \equiv \pm 1 \bmod d$ and $a = (1,1,d-1,d-1)$.
For $t \not =0,1$, consider the curve $X_t: y^d=x(x-1)(x-t)^{d-1}$ over $k=\bar{\FF}_p$.
Then $\aaa_{X_t}=a(X_t)$ and $a(X_t)$ is even.
\end{lemma}

\begin{proof}
By Lemma~\ref{L11hyp}, $X$ has the equation $Y^2=f(W)$ where $f(W) =W^{2d}+(2-4t)W^d +1$.
If $p \nmid 2d$, then $f(W)$ is separable for $t \not = 0,1$.
The Cartier--Manin matrix $M$ is a $g \times g$ matrix where $g=d-1$.
By \cite[page~381]{Yui}, there is a basis for $H^0(X, \Omega^1)$
for which the $(i,j)$th entry of $M$ equals the coefficient $c_{pi-j}$ of $W^{pi-j}$ in $f(W)^{(p-1)/2}$, which is a polynomial in $t$.

Because $f(W)$ is a polynomial in $x=W^d$, note that $c_{pi-j}=0$ unless $d$ divides $pi-j$.
If $p \equiv 1 \bmod d$ (resp.\ $p \equiv -1 \bmod d$) then $M$ is diagonal (resp.\ anti-diagonal).
Up to sign, $\mathrm{det}(M)$ is the product of its non-zero entries. 
For $t \not = 0,1$, the $a$-number $a(X_t)$ is the number of these having $t$ as a root, 
while $\aaa_{X_t}$ is the order of vanishing of $\mathrm{det}(M)$ at $t$.  
By \cite[Lemma~2.3]{hasegawa}, when $p \equiv \pm 1 \bmod d$, 
each non-zero $c_{pi-j}$ has distinct roots.
It follows that $\aaa_{X_t}=a(X_t)$.
Finally, $a(X_t)$ is even by Lemma~\ref{L11hypaa}(2), since the $a$-numbers of $Z_{+1, t}$ and $Z_{-1,t}$ 
are the same.  
\end{proof}

\begin{remark}
After a conversation with Beukers, we think that Lemma~\ref{Laaa=a2} is true for all $p \nmid 2d$.
The reason is that the congruence $j \equiv pi \bmod d$ has a unique solution with $1 \leq j \leq d-1$. 
So every row (or column) of $M$ contains exactly one non-zero entry.
Writing $n=(p-1)/2$ and $pi-j=Kd$, then $c_{pi-j}$ equals the coefficient of $x^K$ in $(x^2+(2-4t)x+1)^n$.
One can show that
$c_{pi-j} \equiv \sum_{k=0}^{K} \binom{j/d}{k} \cdot \binom{1-j/d}{k} t^k \bmod p$.
These polynomials are the unique modulo $p$ polynomial solutions of degree less than $p$ 
of the hypergeometric function with parameters $j/d$, $1-j/d$, and $1$.
They are separable since they are the solutions of a second order linear differential equation, 
which shows $\aaa_{X_t} = a(X_t)$.  This value is even because of the symmetry between $j$ and $d-j$.  
We leave the details of this to an interested reader.
\end{remark}

\begin{remark}  \label{Rhypergeometric}
To determine the exact value of $a(X_t)$, it is necessary to determine whether 
 the $(d-1)/2$ different polynomials in $t$ have roots in common. 
This seems to be an open problem about hypergeometric functions.
\end{remark}



\subsection{Comparison of genus $2$ case with earlier work} \label{SikoA}

The $d=3$ case of Corollary~\ref{CcaseB} is 
compatible with earlier work of \cite{IKO} and \cite{hasegawa}.

Let $p \geq 5$, $d=3$ and $a=(1,1,2,2)$.
Consider the family of genus $2$ curves
\begin{equation} \label{EIKOfirst}
X_t: y^3=x(x-1)(x-t)^2.
\end{equation}

By Corollary~\ref{CcaseB}, the mass formula for this family is $\mu(\cF, p) = (p-1)/(4\cdot 9)$.
This mass formula is not stated in \cite{IKO} or \cite{hasegawa}.
We explain how it can be deduced from these papers, specifically \cite[Proposition~3.2, Theorem~3.3]{IKO}; 
see also \cite[Theorems~2.6, 2.7]{hasegawa}.  
 
This family \eqref{EIKOfirst} of curves $X_t$ of genus $g=2$ is characterized by having $S_3 \subset {\rm Aut}(X_t)$. 
In \cite[Section~1.3]{IKO}, the authors parametrize this family in a different way: 
for $u \not = 0,1$, 
\begin{equation} \label{EIKOfirstalt}
C_u:y_1^2=(x_1^3-1)(x_1^3-u).
\end{equation}

\begin{remark} \label{familynonordsupspec}
By \cite[Proposition~1.10]{IKO}, if $C_u$ is not ordinary, then it is \emph{superspecial}, meaning that 
$J_{C_u}$ is isomorphic to a product of $2$ supersingular elliptic curves.
This implies that $J_{C_u}$ is \emph{supersingular}, meaning that it is isogenous to a product of $2$
supersingular elliptic curves.
On this particular family $C_u$ (or $X_t$), the following properties are all equivalent
(being superspecial, being supersingular, having $p$-rank $0$, and being non-ordinary); this is a very unusual situation.
\end{remark}

\begin{proposition} \cite[Proposition~3.2]{IKO} \label{d3a1122IKO}
Let $p \geq 5$.
The number of isomorphism classes of supersingular curves in the family \eqref{EIKOfirstalt} is
\[N(3, (1,1,2,2), p) = \begin{cases} 
(p-1)/6 & \text{if } p \equiv 1 \bmod 6 \\ 
(p+1)/6 & \text{if } p \equiv 5 \bmod 6.
\end{cases}\]
\end{proposition}

\begin{proof} We briefly sketch the proof.
Let $w=\lfloor (p-1)/3 \rfloor$.
Consider the polynomial
\[G(z) = \sum_{j=0}^w \binom{(p-1)/2}{\lfloor (p+1)/6 \rfloor+j}\binom{(p-1)/2}{j} z^j.\]
The authors prove that the Cartier--Manin matrix of $C_u$ is a scaling of a diagonal or anti-diagonal
$2 \times 2$ matrix by the constant $G(u)$.
This shows that $C_u$ is not ordinary (in fact, superspecial) if and only if 
$u$ is a root of $G(z)$, \cite[Proposition~1.8]{IKO}.
Using Igusa's strategy,
they prove that the roots of $G(z)$ are distinct, and $u=0$ and $u=1$ are not roots
\cite[Proposition~1.14]{IKO}. 
The number of values of $u$ such that $C_u$ is supersingular is thus ${\rm deg}(G(x))=w$. 

The value $u=-1$ is handled separately because $C_{-1}$ has extra automorphisms. 
When $u =-1$, then $C_{-1}$ is supersingular if and only if $p \equiv 5 \bmod 6$, \cite[Proposition~1.11]{IKO}.
By \cite[Lemma~1.5]{IKO}, $C_{u_1} \simeq C_{u_2}$ if and only if either $u_2=u_1$ or $u_2 =1/u_1$.  This divides the count by $2$ when $p \equiv 1 \bmod 6$, or when $p \equiv 5 \bmod 6$ and 
$u \not =-1$.
\end{proof}


There is a small error in \cite[Theorem 3.3]{IKO}, because the family \eqref{EIKOfirstalt} does not specialize to 
\cite[case (5)]{IKO}, which consists of curves whose reduced automorphism group is $S_4$.
Once this is corrected, their work yields the same mass formula $(p-1)/(4 \cdot 9)$, as we explain below.

Let $\epsilon_3=1 - \binom{-3}{p}$, which equals $0$ if $p \equiv 1 \bmod 6$ and equals $2$ if $p \equiv 5 \bmod 6$.

\begin{proposition} \label{PcountbyautCaseB}
Let $p \geq 5$.
The information in \cite{IKO} yields the mass formula $(p-1)/(4 \cdot 9)$ for the 
non-ordinary curves in the family
\eqref{EIKOfirstalt}.
\end{proposition}

\begin{proof}
Write $C=C_u$.  Then $C$ is hyperelliptic.  Let $\mathrm{redAut}(C,\tau)$ denote the 
quotient of $\mathrm{Aut}(C,\tau)$ by the subgroup generated by the hyperelliptic involution.
If $C$ is non-ordinary, then its Cartier--Manin matrix is the zero matrix by Remark~\ref{familynonordsupspec}.
Together with Lemma~\ref{Laaa=a2}, this implies that $\alpha_{C}=2$.
So $\mu(\cF,p)=\sum_{[C]} \frac{1}{\#\mathrm{redAut}(C, \tau)}$, where the sum is over the isomorphism classes 
of non-ordinary curves $C$ in the family \eqref{EIKOfirstalt}.

For the family \eqref{EIKOfirstalt}: 
let $R_n$ be the number of isomorphism classes of non-ordinary 
curves $C$ such that $\#{\rm redAut}(C, \tau) =n$; 
let $R_\infty$ be the number of isomorphism classes of non-ordinary 
singular curves in (the closure of) the family in $\overline{\mathcal{M}}_2$. 

The boundary of the family \eqref{EIKOfirstalt} contains a singular curve $S$
composed of the join of two elliptic curves, each having equation $y^2=x^3-1$.
The curve $S$ is supersingular if and only if $p \equiv 5 \bmod 6$.  
So $R_\infty = \epsilon_3/2$.  
Also $\#\mathrm{redAut}(S) = 6^2$, because each elliptic curve has $6$ automorphisms and there is an automorphism
transposing the two elliptic curves.

If $p=5$, there is a unique non-ordinary curve $X$ in the family by \cite[Theorem~3.3(II)]{IKO}; 
an alternative equation for this curve is $y^2=x^5-x$.
It has reduced automorphism group 
$\mathrm{PGL}_2(5) \simeq S_5$ of order $120$. 
The normalizer of a $3$-cycle in $S_5$ has order $12$.
So $\#\mathrm{redAut}(X, \tau) = 12$.
This yields the mass formula $(1/12) + (1/36)$, which equals $(p-1)/(4 \cdot 9)$.

Suppose $p \geq 7$.
By \cite{Igusa}, the only possibilities for $\mathrm{redAut}(C_u)$ are
$S_3$ or $D_{6}$.  The case $S_4$ does not occur - this is a small error in \cite{IKO}.
Every subgroup of order $3$ is normal in $S_3$ and in $D_6$, 
so $\mathrm{redAut}(C_u,\tau)=\mathrm{redAut}(C_u)$.
This shows that $R_n=0$ unless $n=6$ or $n=12$.

When $u = -1$, then $C_u$ specializes to \cite[Case (4)]{IKO}; it has reduced automorphism group $D_{6}$ 
and is supersingular if and only if $p \equiv 5 \bmod 6$.  So $R_{12} = \epsilon_3/2$.

By Proposition~\ref{d3a1122IKO}, the total number of supersingular curves $C_u$ with 
$\mathrm{redAut}(C_u) \simeq S_3$ is $R_6= (p-1)/6 - \epsilon_3/3$.
This equals $(p-1)/6$ if $p \equiv 1 \bmod 6$ and equals $(p-5)/6$ if 
$p \equiv 5 \bmod 6$.

In conclusion, this yields the mass formula:
\begin{equation} 
R_6/6 + R_{12}/12 + R_\infty/36  = 
(p-1)/36 + e_3(-1/9 + 1/12 + 1/36) = (p-1)/36.
\end{equation}
\end{proof}

\section{Another family of hyperelliptic curves with dihedral action} \label{Sanothercase}

Let $d \geq 4$ be even and $a=(1,d/2,d/2,d-1)$.  
We consider the family of curves
\begin{equation} \label{Ecasealt}
X_t:  y^d=x(x-1)^{d/2}(x-t)^{d/2}.
\end{equation}

Over the two branch points $x=1$ and $x=t$, the fiber of $h:X_t \to {\mathbb P}^1$ 
contains $d/2$ points, each having an inertia group of size $2$. 
Write $d_1=d/2$.  By Lemma~\ref{LRH}, the genus of $X_t$ is $d_1$.

\subsection{The mass formula for $a=(1,d/2,d/2, d-1)$}

\begin{corollary}\label{Ccasesecond}
Suppose $d \geq 4$ with $d$ even and $a=(1,d/2,d/2, d-1)$.  
Suppose $p \nmid 2d$.
The mass formula for the
non-ordinary curves in the family $X_t:  y^d=x(x-1)^{d-1}(x-t)^{d/2}$ is 
\[\mu(d,(1,d/2,d/2, d-1),p) = \begin{cases} (p-1)/2^5 & \text{if } d \equiv 0 \bmod 4\\
(p-1)(d^2+4)/(2^7d^2) & \text{if } d \equiv 2 \bmod 4.
\end{cases}\]
\end{corollary}

\begin{proof}
The signature of the $\mu_d$-cover is given by the dimensions $f_j =0$ if $j$ is even 
and $f_j = 1$ if $n$ is odd.
By Proposition~\ref{Pbouw}, 
$X_t$ is ordinary for a generic value of $t$. 
By Lemma~\ref{lem:deg}, $\delta_{d,a} = 4$. 

Write $d_1=d/2$.  Let $d_2 := \mathrm{gcd}(d_1+1, d)$.  Then $d_2 =1$ if $d \equiv 0 \bmod 4$ and
$d_2=2$ if $d \equiv 2 \bmod 4$.
The result follows from Theorem~\ref{thm:secondeq} once we use
Theorem~\ref{ThmCOSgeneral} to compute:
\[{\mathrm{deg}}_{d,a}(\lambda_1) = (d^2 - (2+2d_1^2) + (d^2+2d_2^2))/12d^2 = 
\begin{cases} 1/8 &  \text{if } d \equiv 0 \bmod 4\\
(d_1^2+1)/8d_1^2 &  \text{if } d \equiv 2 \bmod 4.
\end{cases}\]
\end{proof}

For this family, one can check that $X_t$ is hyperelliptic and that 
$\#\mathrm{Aut}(X_t, \tau) = 2d$, unless $t = -1$ in which case $\#\mathrm{Aut}(X_t, \tau) = 4d$.
We do not have results on the $a$-number of $X_t$ in general.

\subsection{Comparison of genus $2$ case with previous work} \label{SikoB}

Let $d=4$ and $a=(1,2,2,3)$.
This case of Corollary~\ref{Ccasesecond} is 
compatible with earlier work of \cite{IKO}.
Consider the family
\begin{equation} \label{EIKOsecond}
X_t:y^4=x(x-1)^2(x-t)^2.
\end{equation} 

By Corollary~\ref{Ccasesecond}, $\mu(\cF, p) = (p-1)/2^5$.
This mass formula is not stated in \cite{IKO}.
We explain how it can be deduced from that paper. 

This family \ref{EIKOsecond} of genus $2$ curves is called Case (3) in \cite{Igusa} and \cite{IKO}.
The curves $X_t$ in the family are characterized by having $D_4 \subset \mathrm{Aut}(X_t)$. 
The family can also be parametrized as:
\[Y^2=X(X^2-1)(X-\lambda)(X-1/\lambda),\]
or, writing $\beta=(\lambda+1)^2/(\lambda-1)^2$, by
\begin{equation} \label{Eusethis}
C_\beta: Y_1^2=X_1(X_1^2-1)(X_1^2-\beta).
\end{equation}

The automorphism $\tau(X_1,Y_1) = (-X_1, iY_1)$ has order $4$ and 
$\tau^2$ is the hyperelliptic involution.

By \cite[Proposition~1.10]{IKO}, if $X_t$ is not ordinary, then it is superspecial, and thus supersingular.

\begin{proposition} \label{a1322total} \cite[Proposition~3.2]{IKO}
Write $p=8k + \epsilon$ where $\epsilon \in \{1,3,5,7\}$.
The number of supersingular curves in the family is $k$ if $\epsilon \in \{1,3\}$
and is $k+1$ if $\epsilon \in \{5,7\}$.
\end{proposition}




Recall that $\epsilon_3=1 - \binom{-3}{p}$. 
Let $\epsilon_1 = 1 - \binom{-1}{p}$. 
Let $\epsilon_2 = 1 - \binom{-2}{p}$. 

\begin{proposition} \label{PcountbyautCasesecond}
Let $p \geq 7$.
The information in \cite{IKO} yields the mass formula $(p-1)/2^5$ for the 
non-ordinary curves in the family \eqref{EIKOsecond}.
\end{proposition}

\begin{proof}
We work with the family \eqref{Eusethis}.
Each curve $C_\beta$ is hyperelliptic.
If $C_\beta$ is non-ordinary, then its Cartier--Manin matrix is the zero matrix by \cite[Proposition~1.10]{IKO};
From this and \cite[Proposition~1.14]{IKO}, $\alpha_{C_\beta}=2$.
Thus 
$\mu(\cF, p) = \sum_{[C]} \frac{1}{\#\mathrm{redAut}(C, \tau)}$, where the sum is over the isomorphism classes 
of non-ordinary curves $C_\beta$ in \eqref{Eusethis}.

Let $R_\infty$ be the number of isomorphism classes of non-ordinary 
singular curves in (the closure of) the family in $\overline{\mathcal{M}}_2$. 
The boundary of the family \eqref{Eusethis} contains a singular curve $S$
composed of the join of two elliptic curves.
These each have equation $y^2=x^3-x$, which is non-ordinary if and only if $p \equiv 3 \bmod 4$.
So $R_\infty = \epsilon_1/2$.  
Also $\#\mathrm{redAut}(S) = 4^2$, because each elliptic curve has $4$ automorphisms and there is an automorphism
transposing the two elliptic curves.

Let $R_n$ be the number of isomorphism classes of non-ordinary 
smooth curves $C_\beta$ in \eqref{Eusethis} such that $\#{\rm redAut}(C_\beta) =n$. 
Suppose $p \geq 7$.
By \cite{Igusa}, the only possibilities for $\mathrm{redAut}(C_\beta)$ are
$C_2 \times C_2$, $D_{6}$, or $S_4$.
This shows that $R_n=0$ unless $n=4,12, \text{or } 24$.
In the latter two cases, there is an automorphism of order $3$ that does not normalize $\langle \tau^2 \rangle$.

When $\mathrm{redAut}(C_\beta) \simeq D_6$,  
we can identify $\tau$ with a reflection in $D_6$; 
the normalizer of $\langle \tau \rangle$ in $D_6$ is the Klein-4 group, of order $4$.  
When $\mathrm{redAut}(C_\beta) \simeq D_6$, this shows that $\#\mathrm{redAut}(C_\beta, \tau)=4$. 
When $\mathrm{redAut}(C_\beta) \simeq S_4$, 
we can identify $\tau$ with a $2-2$ cycle in $S_4$. 
The normalizer of $\langle \tau \rangle$ in $S_4$ has order $8$.
When $\mathrm{redAut}(X) \simeq S_4$,
this shows that $\#\mathrm{redAut}(X, \tau)=8$. 

\smallskip

{\bf Claim:} $R_4= k - \epsilon_3/2$; $R_{12} = \epsilon_3/2$; and $R_{24} = \epsilon_2/2$. 

\smallskip

{\bf Proof of claim:} Following \cite{IKO}, we see that:

The curve $C_\beta$ has $\mathrm{redAut}(C_\beta)=D_6$ (called Case (4)) if and only if $\beta=9$ (i.e., $\lambda = 2$);
By \cite[Proposition~1.11]{IKO}, $C_9$ is supersingular if and only if $p \equiv 5 \bmod 6$.  
So $R_{12} = \epsilon_3/2$.

The curve $C_\beta$ has $\mathrm{redAut}(C_\beta)=S_4$ (called Case (5)) 
if and only if $\beta=-1$ (i.e., $\lambda = i$);
By \cite[Proposition~1.12]{IKO},
$C_{-1}$ is supersingular if and only if $p \equiv 5, 7 \bmod 8$.  So $R_{24} = \epsilon_2/2$.

By Proposition~\ref{a1322total}, the total number of supersingular curves in the family \eqref{Eusethis} is 
$k + \epsilon_2/2$.
Thus the number such that $\mathrm{redAut}(C_\beta) \simeq C_2 \times C_2$ is $k - \epsilon_3/2$.
This equals the formula found in \cite[Theorem~3.3]{IKO}, 
$R_4=\frac{p-1}{8} - \frac{\epsilon_1}{8} - \frac{\epsilon_2}{4} - \frac{\epsilon_3}{2}$.
In conclusion, this yields the mass formula:
\[\frac{R_4}{4} + \frac{R_{12}}{4} + \frac{R_{24}}{8} + \frac{R_{\infty}}{16} 
= \frac{R_4}{4} + \frac{\epsilon_3/2}{4} + \frac{\epsilon_2/2}{8} + \frac{\epsilon_1/2}{16} = \frac{p-1}{2^5}.\]
\end{proof}

\section{Other work} \label{Srelated}

In \cite{LMS1}, the authors study the $\mu$-ordinary
Newton polygon and Ekedahl--Oort type for a family of cyclic covers of $\mathbb{P}^1$, using the Hasse--Witt triple.  
Under certain conditions, such as having at most $n=5$ branch points, this generalizes the work \cite{Bouw}
about the generic $p$-rank on the family.  In \cite{LMS2}, they continue with an investigation of the 
non-$\mu$-ordinary locus.

There are some other results on mass formulas in positive characteristic, 
which are not closely connected with this paper.
The papers \cite{Yureal}, \cite{YuYu}, \cite{KaremakerYoYu} provide 
mass formula for supersingular abelian varieties of dimensions 2 and 3 and 
supersingular abelian varieties with real multiplication.
These papers take an adelic perspective, building on the work of Ekedahl \cite{ekedahl87}, 
and study the mass of arithmetic quotients of certain double coset spaces.
Our paper generalizes the Eichler--Deuring formula to a different 
class of abelian varieties using a different kind of proof.  

In \cite{Yureal},   
Yu studies the basic locus of Shimura varieties of PEL-type, as defined by Kottwitz \cite{kottwitz1}.
He proves a comparison formula between the geometric mass $\sum \frac{1}{\#\mathrm{Aut}(A)}$ of the basic locus
and the arithmetic mass of a double coset space, but states that it is difficult to compute either one.
For families of $\mu_d$-covers of ${\mathbb P}^1$, the image under the Torelli morphism 
is contained in a Shimura variety $S$ of PEL-type for $\QQ[\mu_d]$.
The results in this paper do not follow from \cite[Theorem~4.6]{Yureal} because 
the Torelli locus typically has positive codimension in $S$.

\bibliographystyle{amsalpha}
\bibliography{supersingular}

\end{document}